\documentclass[12pt, conference]{IEEEtran}

\IEEEoverridecommandlockouts

\onecolumn
\usepackage{cite}
\usepackage{amsmath,amssymb,amsfonts}
\usepackage{algorithmic}
\usepackage{graphicx}
\usepackage{textcomp}
\usepackage{xcolor}
\usepackage{mathrsfs}
\usepackage{algorithmic}
\usepackage{algorithm}
\usepackage{amsmath}
\usepackage{amsthm}
\usepackage{graphicx}
\usepackage{amssymb}
\usepackage{epstopdf}
\usepackage{enumerate}
\usepackage{longtable,tabularx,float}
\usepackage{mathcomp}
\usepackage{multirow}
\usepackage{supertabular}
\usepackage{stmaryrd}
\usepackage{color}
\usepackage{url}
\usepackage{makecell}
\usepackage[OT2,OT1]{fontenc}
\usepackage{bm}
\usepackage{bbm}
\usepackage{caption}
\def\BibTeX{{\rm B\kern-.05em{\sc i\kern-.025em b}\kern-.08em
    T\kern-.1667em\lower.7ex\hbox{E}\kern-.125emX}}

\newtheorem{corollary}{Corollary}[section]
\newtheorem{definition}{Definition}[section]
\newtheorem{lemma}{Lemma}[section]

\newtheorem{theorem}{Theorem}[section]
\newenvironment{thmbis}[1]
 {\renewcommand{\thetheorem}{\ref{#1}}%
 \addtocounter{theorem}{-1}%
 \begin{theorem}}
 {\end{theorem}}

\newtheorem{proposition}{Proposition}[section]
\newtheorem{example}{Example}[section]
\newtheorem{construction}{Construction}[section]
\newtheorem{fact}{Fact}[section]
\newtheorem{remark}{Remark}[section]

\begin{document}

\title{Swap-Robust and Almost Supermagic Complete Graphs for Dynamical Distributed Storage
}

\author{Xin~Wei,~Xiande~Zhang,~and~Gennian~Ge
        \thanks{X. Wei ({\tt weixinma@mail.ustc.edu.cn}) is with the School of Mathematical Sciences, University of Science and Technology of China, Hefei, 230026, Anhui, China.}

\thanks{X. Zhang ({\tt drzhangx@ustc.edu.cn}) is with the School of Mathematical Sciences,
University of Science and Technology of China, Hefei, 230026, and with Hefei National Laboratory, University of Science and Technology of China, Hefei 230088, China.  The research of X. Zhang is supported by the NSFC
under Grants No. 12171452 and No. 12231014, the Innovation Program for Quantum Science and Technology
(2021ZD0302902) and the National Key Research and
Development Program of China (2020YFA0713100).}

\thanks{G. Ge ({\tt gnge@zju.edu.cn}) is with the School of Mathematical Sciences, Capital Normal University, Beijing, 100048, China. The research of G. Ge was supported by the National Key Research and Development Program of China under Grant 2020YFA0712100 and Grant 2018YFA0704703, the National Natural Science Foundation of China under Grant 11971325 and Grant 12231014, and Beijing Scholars Program.
}
}
%

%

\maketitle

\begin{abstract}
To prevent service time bottlenecks in distributed storage systems, the access balancing problem has been studied by designing almost supermagic edge labelings of certain graphs to balance the access requests to different servers. In this paper, we introduce the concept of  \emph{robustness} of edge labelings under limited-magnitude swaps, which is important for studying the \emph{dynamical} access balancing problem  with respect
to changes in data popularity. We provide upper and lower bounds on the robustness ratio for complete graphs with $n$ vertices, and construct $O(n)$-almost
 supermagic labelings that are asymptotically optimal in terms of the robustness ratio.
\end{abstract}


\section{Introduction}
In distributed data storage systems (DSS), data is stored
across a network of nodes. The existing DSS classifies data into hot,
warm and cold objects based on their popularity~\cite{Cherkasova2004,rawat2016locality,joshi2014delay}, and use different
mechanisms to store them.
An important problem in this regime is access balancing, which aims to balance the access requests to different nodes by using data chunk popularity information. Too many access requests on a single node may lead to service time bottlenecks \cite{dean2009challenges}, so it is necessary to balance the popularity of the data when designing the storage system.

Recently in \cite{maxminsum2018},  Dau and Milenkovic proposed the access balancing problem on fractional repetition distributed storage coding \cite{FRCsource2010}. They introduced the MaxMin model
 to maximize the minimum overall popularity of items stored on each server
and thereby control the discrepancy of access requests \cite{chee2020access,Colbourn2021,colbourn2021egalitarian}. These systems are constructed
by selecting a certain combinatorial design, and associating data
items of given popularities with elements of the
design. Later, Yu et al. \cite{Yu2021} introduced another combinatorial model, called MinVar
model for fractional repetition (FR) codes based on graphs, to minimize the variance among the  overall popularities of different servers. In their model, an FR code for DSS is designed from a specific graph, and  data
items of given popularities are associated with edges of the graph.  Since complete graphs, Tur\'an graphs and graphs with large girth were used to construct optimal FR codes in term of file size \cite{FRCsource2010,Turan_Ski2015}, Yu et al. studied an edge labeling problem for such graphs to balance the access requests of these systems \cite{Yu2021}. It was shown that the edge labeling problem for access balancing is closely related to the supermagic labeling problem in graph theory \cite{ho2001some,stewart_1967,supermagic1997}, which requires that the label sum over edges incident to each vertex is a constant.

The above access-balanced storage systems did not take into account the dynamic changes in data demands, that is, the popularities of data items may change with the time.  In the system with dynamically changing item
popularities, undesirable significantly decreasing or increasing of overall popularities may happen to some server and lead to new service time bottlenecks. To solve this problem, one can redistribute data items in the storage system thoroughly to achieve a rebalancing status. For example, Pan et al.~\cite{trade2022} applied a combinatorial structure, called  balanced trade~\cite{trade1990}, to redistribute data items without altering the counts of certain subcollections of items stored on the same server. Since redistribution of data needs lots of resources, Pan et al. \cite{trade2022} studied balanced trades that are as robust as possible to limited-magnitude
swaps in popularities of items to reduce the frequency of  replacements. They called such a trade swap-robust trade,  which is important for studying the stability of server access frequencies with respect
to changes in data popularity.



Motivated by the work of Dau et al. \cite{maxminsum2018} and Pan et al. \cite{trade2022}, we study the access balancing problem for fractional repetition distributed storage coding with dynamically changing item
popularities.  Roughly speaking, we consider the labelings on data items which have good balancing properties and still have good balancing properties after some limited-magnitude swaps. The mathematical problem is formulated as follows.

An FR code can be presented in the language of a hypergraph $G$, that is, each vertex represents a server, and each hyperedge represents a data item. Then hyperedges incident to a vertex means that the copies of the corresponding data items are stored on this server. Popularities of data items are quantified as labels of hyperedges, which
are directly proportional to their popularities and hence their
access frequencies. It has been shown that the access frequencies of data chunks experimentally obey Zipf law \cite{breslau1999web}, where
the $i$-th most popular chunk has access frequency $1/i^\beta$
for some $\beta > 0$. To simplify our model, we assume that the
labels of hyperedges are the ranks of the data items as in \cite{maxminsum2018} and \cite{Yu2021}, that is, $\{1, 2, \ldots, |E|\}$, where $E$ is the set of hyperedges. Then a distribution of data items among all servers is an edge labeling $t$ of $G$, which is a bijection from $E$ to $\{1, 2, \ldots, |E|\}$.

For the access balancing requirement, we need the sums of labels of hyperedges incident to each vertex are as balanced as possible. If the discrepancy of label sums is at most $\alpha$, then we call the labeling $t$ an \emph{$\alpha$-almost supermagic labeling}. If $\alpha=0$, it reduces to the well-known {super}magic labeling \cite{supermagic1997}, which provides a strongly balanced status.
When popularities of data items are changed by limited-magnitude swaps as in~\cite{trade2022}, we use a swap $\theta$ mapping $t$ to another labeling $t'$ to present it, such that $|t(e)-t'(e)|$ is upper bounded by a magnitude say $p$ for all edges $e$. We call such swaps as \emph{$p$-swaps}. The labeling $t$ is \emph{robust} if $t'$ is an $\alpha'$-almost supermagic labeling for a small integer $\alpha'$ after any $p$-swaps.

Considering the complexity of this problem, we assume that $G$ is a complete graph $K_n$. An FR code constructed from a complete graph is the first kind of optimal FR codes in term of file size \cite{FRCsource2010,Turan_Ski2015}. Moreover, an FR code from a complete graph is also a  minimum-bandwidth regenerating (MBR) code, which achieves one end of the well-known storage-repair bandwidth tradeoff (for example, \cite{2018survey}). So we work on this special kind of FR codes, but  the model can be applied to any FR codes.

Under these assumptions, for a given edge labeling $t$ of $K_n$, we use $R(p,t,n)$ to denote the maximum value of $\alpha'$ mentioned above among all $p$-swaps. A lower value of $R(p,t,n)$ means a better robustness of $t$. So it is very important to construct an $\alpha$-almost supermagic labeling of $K_n$ which has a smaller value of  $R(p,t,n)$ for any $p$.
Since popularities of data items, or equivalently labels of edges, are changed  dynamically with time, the magnitude $p$ will accumulate with time. In order to have a reasonable balanced access requests along with time,  we assume that $p=o(n)$ and $p$ goes to infinity as $n$.  Further, we assume that $\alpha=O(n)$, since all known {super}magic labelings or near {super}magic labelings ($\alpha=0$ or $1$) are not robust (see Section~\ref{sec-old_constructions_I} and Appendix~\ref{sec-old_constructions}). Under these assumptions, we show that for any $\alpha$-almost supermagic labeling $t$ on $K_n$,
 \begin{equation} \label{eqR}
pn+o(pn)<R(p, t, n)<2pn+o(pn).
\end{equation}

To measure the robustness of $t$ asymptotically, we introduce the \emph{$p$-robustness ratio of $t$} as $r(p, t)= \varlimsup_{n\to\infty}{R(p, t, n)}/{2pn}$, which satisfies $\frac 1 2\le r(p, t)\le 1$ by Eq. (\ref{eqR}). We will prove that all known {super}magic labelings or near {super}magic labelings ($\alpha=0$ or $1$) in \cite{Colbourn2021,stewart_1967} are not robust asymptotically, that is, the $p$-robustness ratios are all close to $1$.

To construct $\alpha$-almost supermagic labelings on $K_n$ with $p$-robustness ratio closing to the lower bound $1\slash2$, we introduce some new structures, such as weaving squares and $b$-astray good labelings, to help us to construct good labelings. We construct a labeling $T_{8q}$ on $K_{8q}$ with $p$-robustness ratio $3/4$ and good balanced structure. Using the labeling $T_{8q}$ as a starting point, we can recursively prove the following main result in this paper.

\vspace{0.2cm}
\noindent\textbf{Main Result:} For any real number $\varepsilon>0$, there exists a $7n$-almost supermagic labeling construction $t$ for any large enough $n$ such that $r(p, t)\le \frac 1 2+\varepsilon$ for any $p=p(n)$ satisfying $p=o(n)$ and $\lim_{n\to\infty}p=\infty$.
\vspace{0.2cm}
\subsection{Our Contributions}
\begin{itemize}
\item The problem of access balancing fractional repetition distributed storage coding with dynamically changing popularities of items is proposed. To the best of our knowledge, this is the first time for this problem to be raised and  considered systematically.
\item For simplicity, we only consider the case when FR codes are generated by complete graphs and labels are the ranks of data items. Thus the problem is modeled to find an almost supermagic labeling of complete graphs which is robust under $p$-swaps.
For any supermagic labeling $t$ on $K_n$, we give upper and lower bounds for the maximum discrepancy over all $p$-swaps of $t$, i.e., $R(p, t, n)$, in a wide range of parameters.
\item We show that when $n$ goes to infinity, the $p$-robustness ratio of $t$, i.e., $r(p, t)$, is always in the range $[1\slash2, 1]$. Some criteria for bad $p$-robustness ratios are given.
 From these criteria, we prove that all known supermagic or almost supermagic labelings in \cite{Colbourn2021,stewart_1967} have the worst $p$-robustness ratios 1.
\item The most important result (our \textbf{Main Result}) is a recursive construction of a $7n$-almost supermagic labeling with $p$-robustness ratio reaching the lower bound $1\slash 2+\epsilon$ for any $\epsilon>0$.
 During the constructive proof, a  $\frac12n$-almost supermagic labeling $T_{8q}$ over $K_{8q}$ with robustness ratio $3\slash4$ is constructed directly. Three recursive constructions are established.
\item Some partial results for swap-robust and almost supermagic labelings on Tur\'{a}n graphs are also given.
\end{itemize}
\subsection{Organization}
The rest of this paper is structured as follows. In Section~\ref{Sec-preliminaries}, we formally define all necessary notations and state the mathematical problem in detail.  Section~\ref{sec-overview} gives an overview of the paper, where important theorems with sketches of proofs, and main constructions with ideas are involved. Section~\ref{sec-lower_bound} serves to give upper and lower bounds of  $R(p,t,n)$ for general labeling $t$. In Section~\ref{sec-old_constructions_I}, we develop some methods to estimate the $p$-robustness ratio $r(p, t)$ for labelings satisfying certain conditions. Based on these methods, we show that the $p$-robustness ratios of all known {super}magic labelings or near {super}magic labelings in \cite{Colbourn2021,stewart_1967} are close to 1 in Appendix~\ref{sec-old_constructions}. Sections~\ref{sec-main_construction} and~\ref{sec-main_result} are devoted to construct $\alpha$-almost supermagic labelings on $K_n$ with $p$-robustness ratio closing to the lower bound $1/2$. The base labeling $T_{8q}$ is constructed in Section~\ref{sec-main_construction}, while the recursive method is applied in Section~\ref{sec-main_result} for giving a proof to the \textbf{Main Result}. Finally, we conclude our results in Section~\ref{conc.}, where some open problems are also listed.

\section{Problem Statement and Preliminaries}\label{Sec-preliminaries}
We start by defining some basic concepts from graph theory. For other terminologies not mentioned here, we refer the readers to \cite{diestel2005graph,2018survey}.
Let $[a, b]$ denote the set $\{a, a+1, \ldots, b\}$ for any integers $a\le b$, and $[a]\triangleq[1, a]$ for short.



Given a graph $G=(V, E)$ and a vertex $v\in V$, {let $D_G(v)$ be the set of all edges in $E$ containing $v$. The size of $D_G(v)$ is called the degree of $v$ in $G$, denoted by $d_G(v)$. If there is no confusion, we omit the subscript $G$ to write $D(v)$ and $d(v)$ instead.}
A $\delta$-factor of $G$ is a spanning subgraph of $G$ such that all vertices have degree $\delta$. For example, a perfect matching is a $1$-factor, and a Hamiltonian cycle is a $2$-factor. A set of subgraphs $G_1,G_2,\ldots, G_b$ is called a decomposition of $G$, if $E(G_i), i\in [b]$ form a partition of $E(G)$. If each $G_i$ is a $\delta$-factor, then we call it a $\delta$-factor decomposition of $G$.

We consider edge labelings of a graph $G=(V, E)$, which are bijections from $E$ to $[|E|]$.
Given a labeling $t$ and a set $F\subset E$,  the {\it label set} of $F$ is  the set of all values of edges in $F$ under $t$, denoted by $S(t, F)$, and  the {\it label sum} of $F$  is $s(t, F)= \sum_{e\in F}t(e)$. If $F=D(v)$ for some vertex $v$, then we say they are label set and label sum of $v$ instead, and write  $S(t, v)$ and $s(t, v)$ for simplicity. If for any two vertices $u$ and $v$ in $V$, $|s(t,u)-s(t,v)|\le \alpha$ for some integer $\alpha\ge 0$, then we say $t$ is $\alpha$-almost supermagic on $G$. If $\alpha=0$, this reduces to the well-known supermagic labeling in graph theory \cite{stewart_1967}. A graph is called $\alpha$-almost supermagic if there exists an $\alpha$-almost supermagic labeling on it.

Let $K_n$ denote a complete graph on $n$ vertices, and let $K_{m_1, m_2, \ldots, m_r}$ denote a  complete $r$-partite graph whose parts have sizes $m_1, m_2, \ldots, m_r$. If each $m_i=m$, then we write it as $K_{r[m]}$ for short. A Tur\'an graph $T(r, n)$ is an $n$-vertex complete $r$-partite graph such that all parts are of sizes either $\lfloor\frac n r\rfloor$ or $\lceil\frac n r\rceil$.  The union of $m\ge 1$ vertex disjoint copies of a graph $G$ is denoted by $mG$. Throughout this paper, we always let $\epsilon_n:=\binom{n}{2}$ denote the number of edges of $K_n$, and omit the subscript if there is no confusion.


\begin{theorem}\label{1997_structure_theorem}{\rm \cite{stewart_1967, Colbourn2021, supermagic1997}} 
\begin{itemize}
\item[(i)] For all $n\ge 6$, $K_n$ is $1$-almost supermagic. It is further supermagic if and only if $n\not\equiv 0 \mod 4$.
\item[(ii)] The graph $mK_{k[2]}$ is supermagic if and only if $k\ge 3$ and $m\ge 1$.
\item[(iii)] The graph $mK_{k[n]}$ with $n\ge 3$ is supermagic if and only if $k\geq 2$ and one of the following cases is satisfied:
\begin{itemize}
\item[(1)] $n\equiv 0\mod 2$, $m\ge 1;$
\item[(2)] $n\equiv 1\mod 2$, $k\equiv 1\mod 4$, $m\ge 1;$
\item[(3)] $n\equiv 1\mod 2$, $k\equiv 2,3\mod 4$, $m\equiv 1\mod 2$.
\end{itemize}
\end{itemize}
\end{theorem}

It is shown that an $\alpha$-almost supermagic labeling of a Tur\'an graph gives an FR code with good access balancing property in DSS \cite{maxminsum2018, Turan_Ski2015}, where the labels of edges stand for the popularities of data chunks.  Considering the dynamical DSS where popularities may change with time, we introduce the following concept of robustness of an edge labeling.

For an $\alpha$-almost supermagic labeling $t$ on a graph $G=(V, E)$, a swap $\theta$  maps $t$ to another edge labeling of $G$. Here, the edge labels under $t$ measure the original popularities of chunks, and those under $\theta t$ measure the new popularities after the changes with time.  We say a swap $\theta$ is of {\it magnitude} $p$, or \emph{$p$-swap} for short, if for any $e\in E$, $|t(e)-\theta t(e)|\le p$. We care about the value $\alpha'$ such that $\theta t$ is an $\alpha'$-almost supermagic labeling on $G$.
The \emph{$p$-robustness} of  $t$ on $G$, denoted by $R(p, t, G)$, is  the largest $\alpha'$  among all $p$-swaps $\theta$ such that $\theta t$ is $\alpha'$-almost supermagic, that is,
$$R(p, t, G)=\max_{p\text{-swap }\theta}\{\max_{u\neq v}|s(\theta t, u)-s(\theta t, v)|\}.$$
For a given graph $G$ and a magnitude $p$, a smaller value of $R(p, t, G)$ means a better $p$-robustness of the labeling $t$.
We write $R(p, t, n)$ instead of $R(p, t, K_n)$ for short.

Based on the above definitions, it is interesting to construct supermagic labelings with a small $p$-robustness. However,  supermagic labelings  are rare for complete graphs, and all known supermagic labelings in \cite{stewart_1967, Colbourn2021 } are not good regarding to the robustness (see Section~\ref{sec-old_constructions_I} and Appendix~\ref{sec-old_constructions}). So we consider $\alpha$-almost supermagic labelings with $\alpha>0$. Since we only consider the complete graph which has $\Omega(n^2)$ edges, then for any edge labeling $t$, the average label sum of each single vertex reaches $\Theta (n^3)$. So it is reasonable to say that an  $\alpha$-almost supermagic labeling with $\alpha=O(n)$ provides an FR code with tolerable good access balancing property. In the sections that follow, we always assume that $\alpha=O(n)$.

\section{Overview}\label{sec-overview}

In this section, we give an overview of the whole work. The numbering of all theorems here is the same as their original numbering in the following sections.

\noindent \textbf{The upper and lower bounds of robustness.}
For any given $\alpha$-almost supermagic labeling on $K_n$ and magnitude $p$, we give the following bounds of robustness under any $p$-swap.

\begin{thmbis}{bou}
  For any $\alpha$-almost supermagic labeling $t$ on $K_n$,
\[(2n-4)p+\alpha\ge R(p, t, n) \ge (n-2p-19)p-\alpha.\]
\end{thmbis}

\begin{proof}[Sketch of proof of Theorem~\ref{bou}]
The upper bound is simply from that $|D(u)\triangle D(v)|=2n-4$, where $\triangle$ stands for the symmetric difference. For the lower bound, the idea is to directly construct a $p$-swap $\theta=\theta(t)$ such that there exist two different vertices $u$ and $v$ satisfying $|s(\theta t, u)- s(\theta t, v)|\ge (n-2p-19)p-\alpha.$

A naive attempt is to fix two arbitrary vertices $u\ne v$ and then greedily define a mapping $\theta$ like this,
\[
\theta t(e)=
\begin{cases}
t(e)+p, & e\in D(u)\backslash\{uv\}\text{ or }t(e)+p\in S(t, D(v)\backslash \{uv\});\\
t(e)-p, & e\in D(v)\backslash\{uv\}\text{ or }t(e)-p\in S(t, D(u)\backslash \{uv\});\\
t(e), & \text{otherwise.}
\end{cases}
\] That is, increase labels of edges incident to $u$, decrease labels of those incident to $v$, and change labels of those that are in direct conflict with this modification.
However, simply doing this will not yield a $p$-swap $\theta$. For example, if there exist two edges $e_1\in D(u), e_2\in D(v)$ that are not $uv$ and $t(e_1)+2p=t(e_2)$, then by the definition, $\theta t(e_1)=\theta t(e_2)=t(e_1)+p$, contradicting to $\theta$ being a swap. Another example is when there exist two edges $e_1, e_2$ in $D(u)$ such that $t(e_1)+p=t(e_2)$, and another edge $e'$ with  $t(e')=t(e_2)+p$. Then by definition,  $\theta t(e_1)=\theta t(e')= t(e_2)$, contradicting to that $\theta$ is a swap.

Those pairs of edges which may cause problems in construction of $\theta$ are called bad pairs. Several types of bad  pairs are defined regarding to different types of errors, see $U_{\pm}$ and $B_{i}$ in the beginning of Section~\ref{sec-lower_bound}. We should avoid changing the values of edges presented in the bad pairs when constructing $\theta$. The construction is in two steps. Firstly, by double counting we can choose a pair of good vertices such that very few edges in $D(u)\triangle D(v)$ appear in some bad pair (Lemma~\ref{U+U-sum}). Then we can use the construction of $\theta$ as above by avoiding those bad edges and get a $p$-swap with high label sum discrepancy (Lemma~\ref{lemma_robust_lower_bound}).
\end{proof}

 \vspace{0.3cm}

As explained in Introduction,  we always assume that the magnitude $p$ is  \emph{admissible}, that is, satisfies  $p=p(n)=o(n)$ and $\lim_{n\to\infty}p(n)=+\infty$.
 Since $\alpha=O(n)$, we have $\alpha=o(pn)$. To measure the robustness of $t$ asymptotically, we define the \emph{$p$-robustness ratio of $t$} as \[r(p, t)= \varlimsup_{n\to\infty}{R(p, t, n)}/{2pn},\] which satisfies $\frac 1 2\le r(p, t)\le 1$ by Theorem~\ref{bou}.
 A smaller value of $r(p, t)$ means a better robust property of $t$. Our aim is to construct  $\alpha$-almost supermagic labelings on $K_n$ with $\alpha=O(n)$ and $p$-robustness ratios closing to $1/2$ for all admissible $p$.

\vspace{0.3cm}

\noindent\textbf{The robustness ratio and long 1-APs.} Surprisingly, we find that the value of $p$-robustness ratio has strong connection with the density of long 1-APs (a set of consecutive integers) in the label set of each vertex. Roughly speaking, if for each vertex $v$, the long $1$-APs consist of a large fraction of $S(t, v)$ for some $t$, then $t$ has low $p$-robustness ratio. For a given labeling $t$, vertex $v$ and magnitude $p$, we say $S(t, v)$ is \emph{of $(m, \ell)_p$ type for some integers $m$ and $\ell$}, if there exist at most $m$ disjoint $1$-APs in $S(t, v)$ each of length at least $2p$, such that the sum of their lengths is at least $\ell$. Furthermore, we say those disjoint $1$-APs \emph{witness} an $(m, \ell)_p$ type for $v$. If for each vertex $v$, $S(t, v)$ is of $(m, \ell)_p$ type, then we say the edge labeling $t$ is of $(m, \ell)_p$ type. We have the following theorem.

\begin{thmbis}{cor-ml} Let $m$ be a constant integer, $\ell=\beta n$ with $0<\beta<1$ and $\alpha=O(n)$. If $t$ is an  $\alpha$-almost supermagic labeling  on $K_n$ of $(m, \ell)_p$ type, then $r(p,t)\leq 1-\beta$.
\end{thmbis}

From Theorem~\ref{cor-ml}, to construct a labeling $t$ with low robustness ratio, we only need to ensure that the density of long $1$-APs in the label set of each vertex is high. The following is a proof sketch. For the complete version, see Section~\ref{sec-old_constructions_I}.
\begin{proof}[Sketch of proof of Theorem~\ref{cor-ml}]
Firstly we show that the label sum of a long 1-AP does not change too much after a $p$-swap. More specifically, if an edge set $F$ satisfies that $S(t, F)$ is a 1-AP with length $|F|\ge 2p$, then for any $p$-swap $\theta$ on $t$, we can show that $|s(\theta t, F)-s(t, F)|\le p^2$ (Lemma~\ref{1-AP_bridge_lemma}).

Then consider any vertices $u\ne v$. Note that $t$ is of $(m, \ell)_p$ type. So there are at most $m$ long 1-APs contained in $S(t, u)$ and in $S(t, v)$ with length sums at least $\ell$ for each vertex. Hence, the sum difference from $t$ to $\theta t$ caused by the edges with values in those 1-APs is at most $2mp^2$, while the other edges in $D(u)\triangle D(v)$ each can contribute at most $p$ to the sum difference. As a consequence, $|s(\theta t, u)-s(\theta t, v)|\le |s(t, u)-s(t, v)|+2mp^2  +2pn(1-\beta)$. Since $t$ is $\alpha$-almost supermagic and $\alpha=o(pn)$, $r(p, t)\le 1-\beta$.
\end{proof}

\vspace{0.3cm}

On the other hand, if long 1-APs are too sparse in the label sets of a certain number of vertices, we can  deduce that the $p$-robustness ratio of the labeling performs badly. See the following simplified version of  Lemma~\ref{vertex_link_lemma}.  Recall that $\epsilon=\binom{n}{2}$.

\vspace{0.2cm}
\noindent\textbf{Simplified version of Lemma~\ref{vertex_link_lemma}.}
\textit{Let $m\ge 0$ be an integer and $t$ be an $\alpha$-almost supermagic labeling on $K_n$. If there exist $m$ disjoint 1-APs $I_1, \ldots, I_m$ in $[\epsilon]$ with $|I_i|\ge 4p$ for any $i\in [m]$ and a set $T$ of at least $\log n$ vertices such that for each $v\in T$, the 1-APs $I_1, \ldots, I_m$ nearly cover $S(t, v)$, but $|I_i\cap S(t, v)|\le 2$ for each  $i\in [m]$, then $r(p, t)=1-o(1)$.}

\begin{proof}[Sketch of proof of Lemma~\ref{vertex_link_lemma}]
Consider any vertex $v\in T$. Divide $S(t, v)$ into two parts: the set of values in those 1-APs, {denoted by $S_I$; and the set of values outside those intervals, denoted by $S_O$. By the requirements of $T$, $|S_I|$ is almost $n-1$, while $|S_O|$ is small.

Let $P(v)$ denote the set of pairs in $S(t, v)$ with difference at most $p$, then compute the size of $P(v)$. Since $|S_O|$ is small, almost all such pairs are with both edges in $S_I$. However, each $I_i$ is shattered in $S(t, v)$. So one edge in $S_I$ can be only contained in at most three such pairs, and  $P(v)$ is of a small size.}

By pigeon-hole principle and double counting, for any $h>0$, there exists some magnitude $p^*\in [p-h+1, p]$ such that we can choose two good vertices $v_1\ne v_2$ in $T$, satisfying that the number of edges in $D(v_1)\triangle D(v_2)$  appearing in some bad pair with magnitude $p^*$ is almost bounded by the size of $P(v_1)\cup P(v_2)$, which is small. Then one can use the same process as constructing a $p$-swap in the proof of Theorem~\ref{bou} to prove that $r(p^*, t)$ is near  to 1. The proof ends by the non-decreasing property of $r(p, t)$ with $p$.
\end{proof}

Using Lemma~\ref{vertex_link_lemma}, we can carefully check the supermagic and $1$-almost supermagic labeling constructions listed in~\cite{Colbourn2021} and~\cite{stewart_1967}, and then show that those constructions and interaction constructions are all with bad $p$-robustness ratios near to 1. For details, see Example~\ref{exg} and Appendix~\ref{sec-old_constructions}.

\vspace{0.3cm}

\noindent\textbf{Weaving square and astray good labelings.} As aforementioned, all previously known supermagic and almost supermagic labelings are bad for robustness ratio. Good labelings need to have long 1-APs in the label set of each vertex. We  use a type of squares of integers, called {\it weaving squares}, to help distributing long 1-APs efficiently and evenly into each $S(t, v)$.

A weaving square is defined in Definition~\ref{definition_weaving_matrix}, and see Table~\ref{t2} for an example. Basically, a weaving square of order $4q$ should satisfy the following four requirements. (1) All entries are in $[16q^2]$ and distinct from each other. (2) The row sums and column sums are all the same. (3) There exist two disjoint 1-APs with length $q$ in every row and column. (4) The entries with values in $[8q^2]$ cover exactly half of the entries in each row and column. In fact, the entries in a weaving square can be viewed as the values in a supermagic labeling of $K_{4q,4q}$. Combined with a supermagic labeling on $2K_{2q[2]}$, and a random labeling on a matching with $4q$ edges, we can get a labeling of $K_{8q}$, that is $\tau_{8q}$ in Eq.~(\ref{primary_construction}).  This labeling is of $(2, \frac n 4)_p$ type for any $p\le \frac q 2$ because of the weaving square. Then by Theorem~\ref{cor-ml}, its robustness ratio is at most $0.75$, which is away from $1$ and hence better than all previous constructions.

To extend the labeling of $K_{8q}$ to $K_n$ for all $n$ and for a better robustness ratio near to $1/2$, we establish several recursive constructions in Section~\ref{sec-main_result}, where  a kind of stronger labelings,  {\it $b$-astray good labelings}, is needed for the recursions.
Such a labeling is motivated from \cite[Theorem 4]{stewart_1967} and defined in Definition~\ref{definition_3_astray}. Roughly speaking,  a labeling $t$ on $K_n$ is called astray good if the edge set can be partitioned into two parts: a small part $A$, called astray part, in which the edges can be labeled randomly in a wide range around the central interval of $[\epsilon]$; and the part of all other edges which are labeled very carefully with several restrictions to ensure a strong structure. The latter part can further be partitioned into two parts with equal size: the higher part $H$ consisting of all edges with values larger than those in $A$, and the lower part $L$ of those lower-than-$A$ edges.
If further for each vertex $v$, $|D(v)\cap A|\le b$ for some fixed $b$, then we say $t$ is $b$-astray good. A $b$-astray good labeling is $\frac{b^2}2 n$-almost supermagic, but the converse is not true (see Lemma~\ref{4.5n_almost}).


Our recursive constructions input a $b$-astray good labeling of $K_n$ and output a $b$-astray good labeling of $K_{n'}$ with $n'>n$. One important operation in these constructions is to view $K_n$ as an embedded subgraph of $K_{n'}$, and keep or lift the labels of edges in $L$ or $H$ as a whole (which requires the definition of astray good). If the old labeling has some long $1$-APs in the label set of $L$ or $H$, then the long $1$-APs will be kept in the new labeling. For this, a $1$-astray good labeling $T_{8q}$ on $K_{8q}$ is constructed based on $\tau_{8q}$ in Construction~\ref{const_T8q}, which will be used as an input labeling for our recursive constructions.

\vspace{0.3cm}

\noindent\textbf{Recursive constructions and Main theorem.} Three main kinds of recursive constructions are given in Section~\ref{sec:rec}. Most constructions work for $b$-astray good labelings for a general small constant integer $b$. But for our purpose, we set  $b=3$, which is the best we can do.

The first one in Construction~\ref{const_t1} is to extend the labelings of $K_n$ from $n\equiv 0\mod 8$ to all even $n$, which preserves the $3$-astray good property and the original $(m, \ell)_p$ type. The second one in Construction~\ref{const_odd} is extending the labelings of $K_n$ from  even $n$ to odd $n$, but  only keeps the almost supermagic property and  the $(m, \ell)_p$ type. So starting from the $1$-astray good labeling $T_{8q}$ on $K_{8q}$ given in Construction~\ref{const_T8q}, and applying the above two recursive constructions, we can get an $n$-almost supermagic labeling of $K_n$ for all $n$, but with the same robustness ratio as $T_{8q}$, which is $3/4$.

To reduce the robustness ratio from $3/4$ to $1/2$, we need to improve the $(m, \ell)_p$ type of the labeling after recursion, which is a key ingredient when estimating the robustness ratio by Theorem~\ref{cor-ml}.
The third and the most important construction is given in Construction~\ref{const_t0}, which
 inputs  a 3-astray good labeling $t$ on $K_{4q}$, and outputs a 3-astray good labeling $t'$ on $K_{8q}$; if $q\ge 2p$ and $t$ is of $(m, \ell)_p$ type for some $m$ and $\ell$, then $t'$ is of $(m+2, \ell+2q)_p$ type.


We first sketch Construction~\ref{const_t0} and its proof. Split the complete graph $K_{8q}$ into three pieces: two complete graphs $K_{4q}$, and a complete bipartite graph $K_{4q,4q}$. For the two complete graphs, the constraint of $t'$ on them is just two shifted copies of the input $t$  by shifting the values of the edges in each $A$, $L$, $H$ as a whole to cover different values in the underground interval $[\epsilon_{8q}]$. Then the constraint of $t'$ on the bipartite graph is a shifted version of the weaving square of order $4q$. In the label set of each vertex  under $t'$, there are $m$ disjoint $1$-APs of total length $\ell$ inherited from $t$, and two disjoint $1$-APs of length $q$ from the weaving square, so $t'$ is of $(m+2, \ell+2q)_p$ type. The $3$-astray good property of $t'$ can be obtained by routine verification.


We next analyze the goodness of Construction~\ref{const_t0}. By Theorem~\ref{cor-ml}, if both the limits exist,

%
%
$$r(p, t')\le 1-\lim_{q\to\infty}\frac{\ell+2q}{8q}= \frac 1 2\left(\frac 1 2+(1-\lim_{q\to\infty}\frac \ell {4q})\right).$$
This means doing one step of Construction~\ref{const_t0} can decrease the upper bound of robustness ratio  to the average between the original bound and $1\slash 2$. That is why we say $t'$ has a better type  than $t$.
So for any given $\varepsilon\in (0, 1)$, after doing $s\triangleq \lceil -\log \varepsilon\rceil$ steps of Construction~\ref{const_t0}, one can reduce any $3$-astray good labeling no matter how worse the robustness ratio is to a $3$-astray good labeling with ratio at most $\frac 1 2+\varepsilon$. However, this only works for $n$ with $2^{s+3}\mid n$ since each step doubles the vertex numbers.


To make this work for all $n$, we consider a non-increasing sequence $n_0, n_1, \ldots, n_s$, such that $n_0=n$, $n_1=8\lfloor \frac n 8\rfloor$, and for any $i\in [2, s]$, $n_i=4\lfloor\frac{n_{i-1}} 8\rfloor$. By our definition, all $n_i$, $i\ge 1$ are doubly even, and for any $i\in [s-1]$, either $n_i= 2n_{i+1}$ or $n_i= 2n_{i+1}+4$.
We start from a $3$-astray good labeling $t_s$ on $K_{n_s}$, which is either some $T_{8q}$ from Construction~\ref{const_T8q} or some $t_{(2)}$ from Construction~\ref{const_t1}.
 Then for the $i$th step, $i\in [s-1]$, we define $t_{s-i}$ from $t_{s+1-i}$ as follows. If  $n_i= 2n_{i+1}$, apply Construction~\ref{const_t0}; if $n_i= 2n_{i+1}+4$, first apply Construction~\ref{const_t0}, then apply Construction~\ref{const_t1}. After $s-1$ such rounds, we can get a $3$-astray good $t_1$ of $K_{n_1}$ reaching the desired robustness ratio. The final step from $n_1$ to $n$ goes by the first two constructions depending on the parity of $n$. By following this way,  our {\bf{Main Result}} is proved.


\section{ Bounds on the Robustness }\label{sec-lower_bound}

This section serves to provide bounds on the $p$-robustness of a general labeling of $K_n$. Since the robustness is defined on all $p$-swaps of an edge labeling $t$, to create a legal $p$-swap, we need to take care of the pairs of edges with label difference $p$ (or sometimes, $2p$) under $t$.   The following notations are useful and give a classification for the bad pairs mentioned in the Overview.

Let $(V, E)$ be a complete graph $K_n$ and let $t$ be an edge labeling on it. For any vertex $v$, we call a pair $\{e, e'\}\subset D(v)$ a $v$-bad pair if $|t(e)-t(e')|=p$. Denote the set of all $v$-bad pairs as $B_0(v)$. Further define
\[U_+(t, v)=\{e\in D(v): \exists e'\in D(v) \text{ s.t. } t(e')=t(e)+p\}\] and \[U_-(t, v)=\{e\in D(v): \exists e'\in D(v) \text{ s.t. } t(e')=t(e)-p\}.\]

We claim that for any two different vertices $u$ and $v$, $U_+(t, u)\cap U_+(t, v)=\emptyset$. Otherwise, if there exists one edge $e$ such that $e\in U_+(t, u)\cap U_+(t, v)$, then by definition there exists another edge $e'$ with $t(e')=t(e)+p$ such that $e'\in D(u)\cap D(v)$. This means both $e$ and $e'$ are in $D(u)\cap D(v)$, contradicting to the simplicity of the graph. The same argument also shows $U_-(t, u)\cap U_-(t, v)=\emptyset$. So
\begin{equation}\label{upper_bound_U}
\sum_{v\in V}|U_+(t, v)|\le\epsilon \text{ and }
\sum_{v\in V}|U_-(t, v)|\le\epsilon.
\end{equation}
Recall that $\epsilon=\binom{n}{2}$.

For any pair of vertices $u\ne v$ and two edges $e_1, e_2$ satisfying $u\in e_1$ and $v\in e_2$, we say $\{e_1, e_2\}$ is a $\{u, v\}$-bad pair of type I if $|t(e_1)-t(e_2)|=p$, and of type II if $|t(e_1)-t(e_2)|=2p$.
The set of $\{u, v\}$-bad pairs of type I is denoted by $B_1(\{u, v\})$, and the multiset of disjoint union of $B_1(\{u, v\})$ for all $\{u, v\}$ is denoted by $B_1$. Similarly we can define $B_2(\{u, v\})$ and $B_2$ for type II bad pairs, and further $B(\{u, v\})$ and $B$ for the union of both types of bad pairs (not including $v$-bad pairs).

We claim that the sizes of $B_1$ and $B_2$ are roughly upper bounded by $2n^2$.
In fact, for any pair $\{e_1, e_2\}$ with $t(e_1)-t(e_2)=p$ or $2p$, it is a common bad pair for at most four pairs of vertices. On the other hand, there are $\epsilon -p$ and $\epsilon -2p$ pairs of edges with labeling difference $p$ and $2p$, respectively. So,
\begin{equation}\label{upper_bound_B}
|B_1|\le 4(\epsilon-p)\text{ and }|B_2|\le 4(\epsilon-2p).
\end{equation}

\begin{lemma}\label{lemma_robust_lower_bound}
For any edge labeling $t$ on $K_n$ and some integer $\ell$, if there exist vertices $u\ne v$ such that
\begin{equation}\label{eq1}
  |U_+(t, u)|+|U_-(t, v)|+2|B(\{u, v\})|\le \ell,
\end{equation}
then there exists a $p$-swap $\theta$ on $t$ for any $p$ such that
\begin{equation}\label{eq2}
  |s(\theta t, u)-s(\theta t, v)|\ge p(2n-2p-4-\ell)-|s(t, u)-s(t, v)|.
\end{equation}
\end{lemma}
\begin{proof} Let $t$ be any edge labeling of $K_n$ with two vertices $u$ and $ v$ satisfying Eq.~(\ref{eq1}). Our aim is to construct a $p$-swap $\theta$ on $t$ satisfying Eq.~(\ref{eq2}).

 Denote \[B_u=\{e\in D(u): \text{$\exists e'\in D(v)$ s.t. $\{e, e'\}\in B(\{u, v\})$}\},\] and \[B_v=\{e\in D(v): \text{$\exists e'\in D(u)$ s.t. $\{e, e'\}\in B(\{u, v\})$}\}.\] Note that for any $\{u, v\}$-bad pair $\{e, e'\}$, there is a unique ordering $(e_1, e_2)$ of $\{e, e'\}$ such that $u\in e_1$ and $v\in e_2$.
 So $|B_u|, |B_v|\le |B(\{u, v\})|$.
Since $t$ is bijective between $E$ and $[\epsilon]$, we do not distinguish an edge set  and its label set from now on.
Define \[D'=D(u) \backslash \left(U_+(t, u)\cup B_u \cup \left[\epsilon-p+1, \epsilon \right] \cup \{uv\}\right),\] and \[D''=D(v) \backslash \left(U_-(t, v)\cup B_v \cup [p] \cup \{uv\}\right).\] Notice that $D'\cap D''=\emptyset$ and $|D'|+|D''|\ge 2(n-1)-2p-2-(|U_+(t, u)|+|U_-(t, v)|+2|B(\{u, v\})|)\ge 2n-2p-4-\ell$.

We claim that the three sets of edges,
\begin{equation*}
  \begin{array}{cl}
    E_1 = & D'\cup (D''-p), \\
    E_2 = &  (D'+p)\cup D'', \text{ and }\\
    E_3= & (D(u)\backslash D')\cup (D(v)\backslash D'')
  \end{array}
\end{equation*}
 are pairwise disjoint. For $E_1$ and $E_2$, it suffices to show that   $D'\cap (D'+p)=(D''-p)\cap (D'+p)=(D''-p)\cap D''=\emptyset$. First, if there exists some $e\in D'\cap (D'+p)$, then $e-p\in D'$ and $e\in D'$, which means $e-p \in U_+(t, u)$, contradicting to $e-p\in D'$. Second, if there exists some $e\in (D''-p)\cap (D'+p)$, then $\{e-p, e+p\}$ is $\{u, v\}$-bad and $e-p\in B_u$, contradicting to $e\in D'+ p$. Finally, if there exists some $e\in (D''-p)\cap D''$, then $e+p$ and $e$ are both in $D''$, which means $e+p\in U_-(t, v)$, contradicting to $e+p\in D''$.

For $E_1$ and $E_3$, it suffices to show that $D'\cap (D(v)\backslash D'')=(D''-p)\cap (D(u)\backslash D')=(D''-p)\cap (D(v)\backslash D'')=\emptyset$. The first case is trivial since $D'\subset D(u)$ and $D(u)\cap D(v)= \{uv\}$, but the edge $uv$ has been excluded from $D'$.
  If there exists some $e\in (D''-p)\cap (D(u)\backslash D')$, then $\{e, e+p\}$ is $\{u, v\}$-bad and hence $e+p\in B_v$, contradicting to $e\in (D''-p)$. So $(D''-p)\cap (D(u)\backslash D')=\emptyset$. If there exists some $e\in (D''-p)\cap (D(v)\backslash D'')$, then $e$ and $e+p$ are both in $D(v)$ and hence $e+p\in U_-(t, v)$, contradicting to $e\in (D''-p)$. So $(D''-p)\cap (D(v)\backslash D'')=\emptyset$.

  Using the same analysis, we can also prove that $E_2$ and $E_3$ are disjoint.
Now we construct our $p$-swap $\theta$ on $t$  as follows.
\begin{equation*}
\theta t(e)=\left\{
\begin{array}{ll}
t(e)+p, & e\in E_1;\\
t(e)-p, & e\in E_2;\\
t(e), & otherwise.
\end{array}
\right.
\end{equation*}
Since $E_1\cap E_2=\emptyset$ and $E_1+p=E_2$, $\theta$ is a well defined $p$-swap. Now we compute the differences of the label sums of $u$ and $v$ under $\theta t$.
  \begin{align*}
    s(\theta t, u)-s(\theta t, v) & =
s(\theta t, D')+ s(\theta t, D(u)\backslash D') -s(\theta t, D'')- s(\theta t, D(v)\backslash D'')  \\
     & =
|D'|p+s(t, u)+|D''|p-s(t, v) \\
& = (|D'|+|D''|)p+s(t, u)-s(t, v)\\
& \ge (2n-2p-4-\ell)p-|s(t, u)-s(t, v)|.
  \end{align*}
The second equality is due to that $D'\subset E_1$, $D''\subset E_2$ and $E_3=  (D(u)\backslash D')\cup (D(v)\backslash D'')$ is disjoint with $E_1\cup E_2$.
\end{proof}


The following lemma provides a sufficient condition for  the  existence of a pair of vertices required by Lemma~\ref{lemma_robust_lower_bound}.

\begin{lemma}\label{U+U-sum}
For any edge labeling $t$ on $K_n$, if \[\sum_{u\in V}(|U_+(t, u)|+|U_-(t, u)|)\le nm\] for some real  number $m$, then there exist two different vertices $u, v\in V$ such that
\[|U_+(t, u)|+|U_-(t, v)|+2|B(\{u, v\})| \le m+16.\]

\end{lemma}
\begin{proof}
For simplicity, denote $\mu(t, u)=|U_+(t, u)|+|U_-(t, u)|$ for any vertex $u$, then $\sum_{u\in V}\mu(t, u)\le nm$. Suppose on the contrary, such two vertices do not exist. Then for any $u\ne v$, $\mu(t, u)+ \mu(t, v) +4|B(\{u, v\})|>2m+32$. Adding up this inequality over all vertex pairs, we have
\begin{equation*}
\sum_{uv \in E } \left(\mu(t, u)+ \mu(t, v)+ 4|B(\{u, v\})|\right)> \epsilon  (2m+32).
\end{equation*}
Further, on the left hand side,
\begin{equation*}
\begin{array}{ll}
&\sum_{uv\in E } \left(\mu(t, u)+\mu(t, v)+ 4|B(\{u, v\})|\right)\\
=& (n-1)\sum_{v\in V}\mu(t, v)+4|B_1|+4|B_2|\\
\leq & (n-1)nm+16n(n-1),
\end{array}
\end{equation*}
where the last inequality is from Eq.~(\ref{upper_bound_B}). Since $(n-1)nm+16n(n-1) = \epsilon  (2m+32)$, there is a contradiction.

\end{proof}

Combining Lemmas~\ref{lemma_robust_lower_bound} and \ref{U+U-sum}, we have the following bound of $R(p, t, n)$.

\begin{theorem}\label{bou}
  For any $\alpha$-almost supermagic labeling $t$ on $K_n$,
\[(2n-4)p+\alpha\ge R(p, t, n) \ge (n-2p-19)p-\alpha.\]
\end{theorem}
\begin{proof}
Let $m=n-1$. Then by  Eq.~(\ref{upper_bound_U}), $\sum_{u\in V}(|U_+(t, u)|+|U_-(t, u)|)\le nm$ for any edge labeling $t$. By Lemma~\ref{U+U-sum},  there exists a pair $u, v$ of vertices such that $|U_+(t, u)|+|U_-(t, v)|+2|B(\{u, v\})| \le n+15$. Then by Lemma~\ref{lemma_robust_lower_bound} with $\ell=n+15$, $R(p, t, n) \ge (n-2p-19)p-\alpha.$ The upper bound is trivial since any two different vertices $u, v$ in $K_n$ satisfy $|D(u)\Delta D(v)|=2n-4$.
\end{proof}

%

As explained in the Introduction and Overview,  we always assume that the magnitude $p$ satisfies the following two conditions:
\begin{itemize}
\item[] (O1) $p=p(n)=o(n)$,
\item[] (O2) $\lim_{n\to\infty}p(n)=+\infty$.
\end{itemize}
Such an integer $p$ satisfying (O1) and (O2) are called \emph{admissible}. Since  $\alpha=O(n)$, we have $\alpha=o(pn)$. Then Theorem~\ref{bou} can be rewritten as follows.


\begin{theorem}\label{theorem_bound_rpt}
For any $\alpha$-almost supermagic labeling $t$ on $K_n$, if $\alpha=O(n)$ and  $p$ is admissible, then
$$pn+o(pn)<R(p, t, n)<2pn+o(pn).$$
\end{theorem}


To measure the $p$-robustness asymptotically, we define $r(p, t)= \varlimsup_{n\to\infty}\frac{R(p, t, n)}{2pn}$ as the {\it  $p$-robustness ratio} of a labeling $t$ of $K_n$. By Theorem~\ref{theorem_bound_rpt}, $\frac 1 2\le r(p, t)\le 1$. A smaller value of $r(p, t)$ means a better robust property of $t$. Our aim is to construct  $\alpha$-almost supermagic labelings on $K_n$ with $\alpha=O(n)$ and $p$-robustness ratio closing to $1/2$ for all admissible $p$. So from now on, we always assume that $p$ is admissible and $\alpha=O(n)$.



%
\section{Estimating the Robustness Ratio}\label{sec-old_constructions_I}
In this section, we develop some methods to estimate the $p$-robustness ratio for labelings satisfying certain conditions. These conditions are mainly about the distribution of $1$-APs in the label set of each vertex. Here, a $1$-AP means a consecutive interval of integers in $[\epsilon]$. In Section~\ref{sec-lower}, we show that if the label sets of a small fraction of vertices do not contain long $1$-APs, then the labeling is not good enough for resisting $p$-swaps, that is, the $p$-robustness ratio closes to  $1$. By using this criterion, we are able to show that all known constructions of supermagic labelings of $K_n$ are not robust enough (see Appendix~\ref{sec-old_constructions}).
In Section~\ref{sec-mb}, we show the contrast: if the label set of each vertex contains enough long disjoint $1$-APs, then the $p$-robustness ratio of the labeling can be well controlled.

\subsection{Big Robustness Ratio}\label{sec-lower}

The following lemma says that if the label set of each vertex has no long $1$-APs, then the $p$-robustness ratio reaches the biggest value.
\begin{lemma}\label{slice_lemma}
 For any $\alpha$-almost supermagic labeling  $t$ on $K_n$, if there exists a partition of $[\epsilon]$ into $m$ $1$-APs  $I_1, \ldots, I_m$ with $m=o(n^2\slash p)$, such that for each vertex $v$, $|I_i\cap S(t, v)|\le 1$ for all $i\in [m]$,
then $r(p, t)=1$.
\end{lemma}
\begin{proof}
Let  $\ell_i =|I_i|$, $i\in[m]$. Suppose  $I_i=[x, x+\ell_i-1]$ for some integer $x$. If  an edge $e$ satisfies $x+p\le t(e)<x+\ell_i-p$, then $e$ is not contained in  $U_+(t, v)\cup U_-(t, v)$ for any vertex $v$ due to that $|I_i\cap S(t, v)|\le 1$. Thus, every $I_i$ keeps at least $\ell_i-2p$ edges out of $U_+(t, v)$ and $U_-(t, v)$ for all vertices $v$. Since $\cup_{i=1}^m I_i=[\epsilon ]$, we have $\sum_{v\in V}|U_+(t, v)|\le 2mp$ and $\sum_{v\in V}|U_-(t, v)|\le 2mp$.

By Lemma~\ref{U+U-sum}, there exist different vertices $u$ and $v$ such that $|U_+(t, u)|+|U_-(t, v)|+2|B(\{u, v\})|\le {4mp}/ n+16$. Then by Lemma~\ref{lemma_robust_lower_bound}, there  exists a $p$-swap $\theta$ such that $|s(\theta t, u)-s(\theta t, v)|\ge (2n-2p -20- {4mp} /n)p-\alpha$. So $R(p, t, n)\ge (2n-2p -20- {4mp} /n)p-\alpha$. By $m=o(n^2\slash p)$, $R(p, t, n)\ge 2pn+o(pn)$ and hence $r(p, t)=1$.
\end{proof}

Next, we relax the condition of Lemma~\ref{slice_lemma} by requiring that the label sets of $\log(n)$ vertices contain no long $1$-APs, and give a lower bound on the robustness ratio.

\begin{lemma}\label{vertex_link_lemma}
Let $m\ge 0$ be an integer and $t$ be an $\alpha$-almost supermagic labeling on $K_n$. If there exist $m$ disjoint $1$-APs $I_1, I_2, \ldots I_m$ in $[\epsilon]$ with $|I_i|\geq4p$, $i\in [m]$, and
a set $T$ of at least $\log(n)$ vertices  such that the following hold for any vertex $v\in T$,
\begin{itemize}
\item[(1)] $|S(t, v)\cap (\cup_{i=1}^m I_i)|\ge n-1-\varepsilon\sqrt n$ for some constant $\varepsilon \ge 0$;
\item[(2)] $|S(t, v)\cap I_i|\le 2$ for any $i\in [m]$,
\end{itemize}
then for any integer constant $h>0$, there exists $p^*\in[p-h+1, p]$ such that $r(p, t)\ge r(p^*, t)\ge 1-\frac{(3+\varepsilon^2)}{2h}$.
\end{lemma}
\begin{proof}
Consider any vertex $v\in T$. Divide $S(t, v)$ into two parts: the set of values in those intervals, denoted by $S_I= S(t,v) \cap (\cup_{i=1}^m I_i)$; and the left values out of those intervals, denoted by $S_O= S(t,v)\backslash S_I$. By (1), $|S_O|\le \varepsilon\sqrt n$.

Let $P$ denote the set of all pairs in $ S(t,v)$ with difference at most $p$, and then each pair in  $P$ intersects $S_O$ at most two edges. Let $P_i=\{A\in P: |A\cap S_O|=i \}$ for  $i\in [0,2]$.  Next, we compute the size of each $P_i$. For any $s\in S(t,v)$, denote $I(s)=[s-p, s+p]$. Since each interval $I_i$ has length at least $4p$, the interval $I(s)$ intersects at most two intervals $I_i$. By (2), we have $S_I\cap I(s)\le 4$ for each $s\in S(t,v)$. Then
 \[|P_0|\le \frac12 \sum_{s\in S_I}(|S_I\cap I(s)|-1) \le \frac32 n\]
by considering that for each $s\in S_I$ there are $|S_I\cap I(s)|-1$ numbers in $S_I$ who differ from $s$ at most $p$. Similarly $|P_1|\le \sum_{s\in S_O}|S_I\cap I(s)|\le 4\varepsilon\sqrt n$, and $|P_2|<\frac12 |S_O|^2=\frac{\varepsilon^2}2n$. Then $|P|=|P_0|+|P_1|+|P_2|<\frac{\varepsilon^2}2 n+4\varepsilon \sqrt n +\frac 3 2 n$.


By abuse of notation, when $t$ is explicit, we denote $U_+(t, v)$ and $U_-(t, v)$ under some magnitude $p$ as $U_+(p, v)$ and $U_-(p, v)$. By definition of $P$, we have \[\sum_{p'\le p}U_+(p', v)+\sum_{p'\le p}U_-(p', v)\le 2 |P| <(3+\varepsilon^2)n +8\varepsilon\sqrt n.\] Then by pigeon-hole principle, for any integer $h>0$, there exists some $p'(v)\in [p-h+1, p]$ such that \[U_+(p'(v), v)+U_-(p'(v), v)<\frac{(3+\varepsilon^2)}hn +\frac{8\varepsilon}h \sqrt n.\]

Consider $p'(v)$ for all vertices $v$ in $T$, and at least $\frac {log(n)} h$ of them are coincident. Denote the set of those vertices in $T$ by $T'=\{v_1, v_2, \ldots, v_x\}$ for some integer $x\ge \frac {log(n)} h$, and denote the common magnitude by $p^*$. For $v_1$, there are at most $4n$ bad pairs (of any type) of $\{e, e'\}\subset E(K_n)$ under $p^*$ containing $v_1$, and each pair can be in at most two $B(\{v_1, v\})$ for some $v\ne v_1$. So $\sum_{v\in V\backslash \{v_1\}}|B(\{v_1, v\})|\le 8n$. By pigeon-hole principle, there exists another vertex in $T'$, say $v_2$, such that $|B(\{v_1, v_2\})|\le \frac{8hn}{log(n)}$. Finally, we find out a pair of different vertices $v_1$ and $v_2$ in $T'$, such that under magnitude $p^*$, \[|U_+(t, v_1)|+|U_-(t, v_1)|+ |U_+(t, v_2)|+ |U_-(t, v_2)|+4|B(\{v_1, v_2\})|<2(\frac{(3+\varepsilon^2)}hn +\frac{8\varepsilon}h \sqrt n)+\frac{32hn}{log(n)}.\]

Without loss of generality, assume that $|U_+(t, v_1)|+|U_-(t, v_2)|+2|B(\{v_1, v_2\})| <\frac{(3+\varepsilon^2)}hn +\frac{8\varepsilon}h \sqrt n+\frac{16hn}{log(n)}$ under magnitude $p^*$. By Lemma~\ref{lemma_robust_lower_bound}, $R(p^*, t, n)\ge (2n-4-2p^*-\frac{(3+\varepsilon^2)}hn -\frac{8\varepsilon}h \sqrt n-\frac{16hn}{log(n)})p^*-\alpha$. Since $\varepsilon, h$ are constants, $p$ is admissible, and $\alpha=o(p^*n)$, we have $r(p^*, t)\ge 1-\frac{(3+\varepsilon^2)}{2h}$. The proof is completed since $r(p,t)$ is non-decreasing with $p$.
\end{proof}

The following corollary is a special case of Lemma~\ref{vertex_link_lemma} by setting $T=V$ and $\varepsilon=0$.

\begin{corollary}\label{link_corollary}
Let $m\ge 0$ be an integer and $t$ be an $\alpha$-almost supermagic labeling on $K_n$. If there exists a decomposition of $[\epsilon]$ into $m$ pieces of disjoint $1$-APs $I_1, I_2, \ldots I_m$ with lengths no smaller than $4p$, such that for any vertex $v$, $|S(t, v)\cap I_i|\le 2$ for any $i\in [m]$, then for any integer constant $h>0$, there exists a new magnitude $p^*\in[p-h+1, p]$ such that $r(p^*, t)\ge 1-\frac 3 {2h}$.
\end{corollary}



There are mainly two kinds of constructions for supermagic or almost supermagic labelings on $K_n$ appeared in \cite{stewart_1967, Colbourn2021}. Due to their construction strategies, we call them factorial constructions \cite{Colbourn2021}  and inductive constructions \cite{stewart_1967}. By applying Lemmas~\ref{slice_lemma}, ~\ref{vertex_link_lemma} and Corollary~\ref{link_corollary}, we show that neither constructions are  robust. The following is an example of factorial constructions, while the complete analysis is put in Appendix~\ref{sec-old_constructions}.


\begin{example}\label{exg} The following construction was given in \cite[Lemma 3]{Colbourn2021}. Let $n=4s+2$ for some positive integer $s$, and let the vertex set of $K_n$
be $\mathbb Z_{4s+1}\cup\{\infty\}$. Then $K_n$ has a decomposition into $4s+1$ $1$-factors $\{\{\infty, i\}\}\cup \{\{j+i, 4s+1-j+i\}:j\in[s]\}\cup \{\{j-1+i, 4s+2-j+i\}: j\in[s+2, 2s+1]\}$, $i\in [0,4s]$.
%

For any $i\in [0, 4s]$, assign the labels in $[i(2s+1)+1, (i+1)(2s+1)]$ to edges in the $i$th $1$-factor of $K_n$ as follows:
$$
\begin{cases}
t(\{j+i, 4s+1-j+i\})=j+i(2s+1)\text{ for }j\in[s];\\
t(\{\infty, i\})=s+1+i(2s+1);\\
t(\{j-1+i, 4s+2-j+i\})=j+i(2s+1)\text{ for }j\in[s+2, 2s+1].
\end{cases}
$$
From~\cite{Colbourn2021}, the labeling $t$ is supermagic and the label sum for each vertex is $(4s+1)(4s^2+ 3s+1)$.

Consider the $p$-swap robustness of  $t$ for any admissible $p$.  Set $m=4s+1$. So $m=o(n^2\slash p)$. For any $i\in [4s+1]$, define $I_i=[(i-1)(2s+1)+1, i(2s+1)]$. For any vertex $v$, by the factor decomposition structure, $|I_i\cap S(t, v)|=1$ for any $i\in [m]$, and hence by Lemma~\ref{slice_lemma}, we have $r(p, t)=1$.

\end{example}

\subsection{Long $1$-APs are Robust under $p$-Swaps}\label{sec-mb}

The following lemma shows that if a label set forms a $1$-AP of size at least $2p$, then its label sum  difference changed  by a $p$-swap is at most $p^2$.

\begin{lemma}\label{1-AP_bridge_lemma}
Let $t$ be an edge labeling on a graph $K_n$. If an edge set $F$ with $|F|\ge 2p$ satisfies that $S(t, F)$ forms a $1$-AP, then for any $p$-swap $\theta$ on $t$, $|s(\theta t, F)-s(t, F)|\le p^2$.
\end{lemma}
\begin{proof}Denote $\ell=|F|\geq 2p$. Suppose that $S(t, F)=[x,   y]$ with $y-x+1= \ell$.
For simplicity, we assume that $x-p>0$ and $y+p\le \epsilon$. For any $p$-swap $\theta$, it is clear that $S(\theta t, F)\subset [x-p, y+p]$.

Partition $E(G)$ into three parts, $F$, $E_1$ and $E_2$, where $S(t, E_1)=[x-1]$ and $S(t, E_2)=[y+1,\epsilon]$. Then $S(\theta t, E_1)\subset [x+p-1]$, and $S(\theta t, E_2)\subset [y+1-p, \epsilon]$.
So $|S(\theta t, F)\cap [x-p, x+p-1]|=p$,  $|S(\theta t, F)\cap [y+1-p, y+p]|=p$, and the label set of the rest $\ell-2p$ edges of $F$ under $\theta t$ forms exactly $[x+p, y-p]$.


By our analysis, the possible minimum value of $s(\theta t, F)$ is attained when $S(\theta t, F)\cap [x-p, x+p-1]=[x-p, x-1]$ and $S(\theta t, F)\cap [y+1-p, y+p]=[y+1-p, y]$. So, $s(\theta t, F)\ge\sum_{i=1}^p (x-i) + \sum_{j=x+p}^{y}j = s(t, F)-p^2$. Similarly, the possible maximum value is attained when $S(\theta t, F)\cap [x-p, x+p-1]=[x, x+p-1]$ and $S(\theta t, F)\cap [y+1-p, y+p]=[y+1, y+p]$. So $s(\theta t, F)\le\sum_{i=1}^p (y+i) + \sum_{j=x}^{y-p}j =  s(t, F)+p^2$. From both sides, $|s(\theta t, F)-s(t, F)|\le p^2$.
\end{proof}

By Lemma~\ref{1-AP_bridge_lemma}, in order to construct a labeling which is robust under $p$-swaps, we need to fill each $S(t, v)$ by disjoint long $1$-APs. We give the formal definition below.

\begin{definition}
For a given labeling $t$, vertex $v$ and magnitude $p$, we say $S(t, v)$ is \emph{of $(m, \ell)_p$ type for some integers $m$ and $\ell$}, if there exist at most $m$ disjoint $1$-APs in $S(t, v)$ each of length at least $2p$, such that the sum of their lengths is at least $\ell$. Furthermore, we say those disjoint $1$-APs \emph{witness} an $(m, \ell)_p$ type for $v$. If for each vertex $v$, $S(t, v)$ is of $(m, \ell)_p$ type, then we say the edge labeling $t$ is of $(m, \ell)_p$ type.
\end{definition}

\begin{example}
If $p=2$ and $t$ is a labeling on $K_{10}$ with vertex set $\{v_1, v_2, \ldots, v_{10}\}$ such that for any $e_i\triangleq v_iv_{10}$, $i\in [9]$,
$$t(e_1)=8; t(e_2)=21; t(e_3)=7; t(e_4)=23; t(e_5)=41; t(e_6)=22; t(e_7)=6; t(e_8)=5; t(e_9)=24,$$
then $S(t, v_{10})$ is of $(2, 8)_2$ type with corresponding witness $1$-APs $[5, 8]$ and $[21, 24]$. Moreover, the preimages of those two $1$-APs are $\{e_1, e_3, e_7, e_8\}$ and $\{e_2, e_4, e_6, e_9\}$, respectively.
\end{example}
 Under these definitions, we have the following lemma.

\begin{lemma}\label{1-AP_dominent_lemma}Let $m$, $p$ and $\ell$ be positive integers.
If $t$ is an edge labeling on $K_n$ of $(m, \ell)_p$ type,
then for any different vertices $u$, $v$ and $p$-swap $\theta$,
$$|s(\theta t, u)-s(\theta t, v)|\le |s(t, u)-s(t, v)|+2mp^2+2p(n-\ell-1).$$
\end{lemma}
\begin{proof}
We give the proof for the case when $s(\theta t, u)\ge s(t, u)\ge s(t, v)\ge s(\theta t, v)$. For other cases, the proofs are similar.

Suppose that $1$-APs $I_1, I_2, \ldots, I_{m_1}$ in $S(t, u)$, and $1$-APs $I_1', I_2', \ldots, I_{m_2}'$ in  $S(t, v)$ witness an  $(m, \ell)_p$ type for $u$ and $v$, respectively, for some $m_1,m_2\leq m$.  Define the corresponding sets of edges by $E_i$ and $E_i'$. Then each set is of size at least $2p$, and  each  union $\cup_{i=1}^{m_1}E_i$ and $\cup_{i=1}^{m_2}E_i'$ is of size at least $\ell$. {
Let $D_O=D(u)\backslash(\cup_{i=1}^{m_1}E_i)$.
Then by Lemma~\ref{1-AP_bridge_lemma}
\begin{equation*}
\begin{array}{rl}
s(\theta t, u)=& \sum_{i=1}^{m_1}s(\theta t, E_i)+s(\theta t, D_O)\\
\le & \sum_{i=1}^{m_1}s(t, E_i)+mp^2 +s(t, D_O)+p|D_O|\\
\le & s(t, u)+mp^2+ p(n-\ell-1).
\end{array}
\end{equation*}
By the same analysis, $s(\theta t, v)\ge s(t, v)-mp^2-p(n-\ell-1)$. So, $s(\theta t, u)-s(\theta t, v)\le s(t, u)-s(t, v)+2mp^2+2p(n-\ell-1)$.}
\end{proof}

By Lemma~\ref{1-AP_dominent_lemma}, less long $1$-APs with larger sum length in each $S(t, v)$ leads to better robustness. The following is immediate by setting $m$ a constant and $\ell$ linear with $n$.

\begin{theorem}\label{cor-ml} Let $m$ be a constant integer, $\ell=\beta n$ with $0<\beta<1$ and $\alpha=O(n)$. If $t$ is an  $\alpha$-almost supermagic labeling  on $K_n$ of $(m, \ell)_p$ type, then $r(p,t)\leq 1-\beta$.
\end{theorem}

To construct examples of $\alpha$-supermagic labelings   on $K_n$ of $(m, \beta n)_p$ type for some constant $\beta$, we need a lot of preparatory work. So we defer the constructions in the next section. For example,  the labeling $\tau_{8q}$ defined in Eq.~(\ref{primary_construction}) is an $\frac{n}{2}$-almost supermagic labeling on $K_{8q}$ of  $(2, 2q)_p$ type for any $p\le q/2$.

\section{Direct Constructions}\label{sec-main_construction}

In this section, we give direct constructions of $\alpha$-almost supermagic labelings with bounded robustness ratio. From Theorem~\ref{cor-ml}, labelings of $(m, \ell)_p$ type with constant $m$ and $\ell=\Theta(n)$, that is,  disjoint $1$-APs with large length in each $S(t, v)$, have a $p$-robustness ratio away from $1$. In Subsection~\ref{sec-weaving}, we show how to distribute long $1$-APs to each $S(t, v)$ efficiently by using the so-called weaving squares.  In Subsection~\ref{sec-astray_good}, we introduce a stronger structure than $\alpha$-almost supermagic labeling, the so-called $b$-astray good labeling, and directly construct a $1$-astray good labeling of $(2, 2q)_p$ type. Such a structure is very useful in the recursive constructions developed in the next section.



\subsection{Weaving Squares}\label{sec-weaving}


For any positive integer $q$, we first define a {\it little square} $L_q$ as a $q\times q$ square, such that for any $(i, j)\in [q]\times [q]$, $L_q(i, j)=q(i-1)+j$. Let $L^0_q$, $L^1_q$, $L^2_q$ and $L^3_q$ be four squares obtained by rotating $L_q$ with angels $0$, $\frac \pi 2$, $\pi$ and $\frac {3\pi} 2$, respectively. See Table~\ref{t1} for an example. In fact, we can use an arrow to indicate a little square with that direction, that is, use $\rightarrow$, $\downarrow$, $\leftarrow$, $\uparrow$ to stand for $L^0_q$, $L^1_q$, $L^2_q$ and $L^3_q$, respectively.  One can see that each $L^i_q$ contains $q$ disjoint $1$-APs of length $q$ along with the corresponding direction.
\begin{table*}[h]
\centering
\caption{Four little squares of order $3$}\label{t1}
\begin{tabular}{c|c|c|c|c|c|c|c|c|c|c|c|c|c|c|c|}
\cline{2-4} \cline{6-8} \cline{10-12} \cline{14-16}
 &1&2&3& &7&4&1& &9&8&7& &3&6&9\\
\cline{2-4} \cline{6-8} \cline{10-12} \cline{14-16}
$L^0_3$:&4&5&6&$L^1_3$:&8&5&2&$L^2_3$:&6&5&4&$L^3_3$:&2&5&8\\
\cline{2-4} \cline{6-8} \cline{10-12} \cline{14-16}
 &7&8&9& &9&6&3& &3&2&1& &1&4&7\\
\cline{2-4} \cline{6-8} \cline{10-12} \cline{14-16}
\end{tabular}
\end{table*}

Next, we construct a square $B$ of order four as follows, which we call the {\it base square}. Note that all integers from $1$ to $16$ appear in $B$ exactly once, and the smaller integers ($\leq$ 8) are located in the upper-left and lower-right $2\times2$ subsquares. Further, the row sums and column sums of $B$ are all the same. {\color{red}}
\begin{center}
\begin{tabular}{c|c|c|c|c|}
\cline{2-5}
\multirow{4}*{$B$:}&1&6&11&16\\
\cline{2-5}
 &7&4&13&10\\
\cline{2-5}
 &12&15&2&5\\
\cline{2-5}
 &14&9&8&3\\
 \cline{2-5}
\end{tabular}
\end{center}

Finally,  we define our {\it weaving square} $W_{4q}$ with order $4q$ as follows. In fact, the weaving square $W_{4q}$ is obtained by replacing each entry of $B$ with a suitable little square of order $q$ and lifting each entry of this sub-square with a constant magnitude so that all entries of  $W_{4q}$ are different. The way of inputting the little squares is illustrated as follows, where each arrow means a little square with that direction.

\begin{center}
\begin{tabular}{|c|c|c|c|}
\hline
$\rightarrow$&$\downarrow$&$\leftarrow$&$\uparrow$\\
\hline
$\uparrow$& $\rightarrow$&$\downarrow$&$\leftarrow$ \\
\hline
 $\leftarrow$&$\uparrow$& $\rightarrow$&$\downarrow$\\
\hline
$\downarrow$&$\leftarrow$&$\uparrow$&$\rightarrow$\\
\hline
\end{tabular}
\end{center} We give a formal definition below, and see Table~\ref{t2} for an example.

\begin{definition}\label{definition_weaving_matrix}
For any $(i, j)\in [4q]\times [4q]$, there exist unique integers $i_1, j_1\in[4]$ and $i_2, j_2\in [q]$, such that $i=q(i_1-1)+i_2$, and $j=q(j_1-1)+j_2$. Define
$$W_{4q}(i, j)=(B(i_1, j_1)-1)q^2+L^{(j_1-i_1\mod 4)}_q(i_2, j_2).$$
\end{definition}

\begin{table*}[h]
\centering
\caption{The weaving square $W_{12}$}\label{t2}
\footnotesize{
\begin{tabular}{|c|c|c|c|c|c|c|c|c|c|c|c|}
\hline
1&2&3&52&49&46&99&98&97&138&141&144\\
\hline
4&5&6&53&50&47&96&95&94&137&140&143\\
\hline
7&8&9&54&51&48&93&92&91&136&139&142\\
\hline
57&60&63&28&29&30&115&112&109&90&89&88\\
\hline
56&59&62&31&32&33&116&113&110&87&86&85\\
\hline
55&58&61&34&35&36&117&114&111&84&83&82\\
\hline
108&107&106&129&132&135&10&11&12&43&40&37\\
\hline
105&104&103&128&131&134&13&14&15&44&41&38\\
\hline
102&101&100&127&130&133&16&17&18&45&42&39\\
\hline
124&121&118&81&80&79&66&69&72&19&20&21\\
\hline
125&122&119&78&77&76&65&68&71&22&23&24\\
\hline
126&123&120&75&74&73&64&67&70&25&26&27\\
\hline
\end{tabular}
}
\end{table*}

\begin{lemma}\label{properties_for_weaving_matrix}
For any positive integer $q$, the weaving square $W_{4q}$ satisfies the following properties.
\begin{itemize}
\item[(1)] All integers from $1$ to $16q^2$ appear in $W_{4q}$ exactly once.
\item[(2)] The row sums and column sums all equal $32q^3 +2q$.
\item[(3)] There exist two disjoint $1$-APs with length $q$ in every row and column.
\item[(4)] For any $(i,j)\in [4q]\times[4q]$, $W_{4q}(i, j)\le 8q^2$ if and only if $i,j\le 2q$ or $i, j> 2q$.
\end{itemize}
\end{lemma}
\begin{proof}By the definition of  $W_{4q}$, the properties (1), (3) and (4) are clear. We only prove (2) for the row sums. For each row $i\in [4q]$, there exists a unique pair $(i_1, i_2)\in [4]\times[q]$ such that $i=(i_1-1)q +i_2$. By the definition of $W_{4q}$,
\begin{equation*}
\begin{aligned}
\sum_{j=1}^{4q} W_{4q}(i, j) &=\sum_{j=1}^q W_{4q}(i, j)+ \sum_{j=1}^q W_{4q}(i, q+j)+ \sum_{j=1}^q W_{4q}(i, 2q+j)+ \sum_{j=1}^q W_{4q}(i, 3q+j)\\
 &=\sum_{k=1}^4 \sum_{j=1}^q \left((B(i_1, k)-1)q^2+L_q^{(k-i_1 \mod 4)}(i_2, j)\right)\\
 &=q^3\left(\sum_{k=1}^4 B(i_1, k)-4\right)+\sum_{j=1}^q\left(\sum_{k=1}^4 L_q^{(k-i_1 \mod 4)}(i_2, j)\right).
\end{aligned}
\end{equation*}
Notice that $\sum_{k=1}^4B(j, k)=34$ for each $j\in [4]$, and $\sum_{k=0}^3L_q^{k}(i, j)=2q^2+2$ for any $(i, j)\in [4q]\times [4q]$. So
\begin{equation}\label{rowsum}
                                  \sum_{j=1}^{4q} W_{4q}(i, j)=30q^3+q(2q^2+2)=32q^3 +2q,
                                \end{equation} which does not depend on the value of $i$, hence the row sums are the same.\end{proof}

%

%
%

 Now we give an example of constructing a labeling on $K_{8q}$ by using weaving squares.
Before that, we have the following observations.
\begin{remark}\label{re-dec}
  Suppose that there exists an edge decomposition of a graph $G$ into $m$ factors $G_1, G_2, \ldots, G_m$, where each $G_i$ has a labeling $t_i$. Then we can construct a labeling $t$ from $t_i, i\in [m]$, by assigning labels for edges in $G_i$ sequentially, that is,  $t(e)=t_i(e)+\sum_{j=1}^{i-1}|E(G_i)|$ if $e\in E(G_i)$. It is easy to see that if each $t_i$ is an $\alpha_i$-almost supermagic labeling  on $G_i$, then $t$ is an $(\sum_{i=1}^m\alpha_i)$-almost supermagic labeling on $G$. Further, if each $t_i$ is of $(\lambda_i, \ell_i)_p$ type, then $t$ is of $(\sum_{i=1}^m\lambda_i, \sum_{i=1}^m \ell_i)_p$ type.
\end{remark}

Consider $K_{8q}$ with vertices $v_1,  \ldots, v_{4q},u_1,  \ldots, u_{4q}$,  where $q\ge 2$. Decompose $K_{8q}$ into three subgraphs $G_1, G_2$ and $G_3$ as follows: $G_1$ is a bipartite complete graph with two parts $\{v_1, v_2, \ldots, v_{4q}\}$ and $\{u_1, u_2, \ldots, u_{4q}\}$; $G_3$ is a $1$-factor of $G$ with edge set $\{v_{2i-1}v_{2i},u_{2i-1}u_{2i}: i\in [2q]\}$; and $G_2$ consists of all the remaining edges. Observe that $G_i, i=1, 2, 3$ are factors of $K_{8q}$ with degrees $4q$, $4q-2$ and $1$, respectively. We use $W_{4q}$ to define a labeling $t_1$ of $G_1$, that is, $t_1(u_iv_j)= W_{4q}(i, j)$, for $i,j\in [4q]$. By Lemma~\ref{properties_for_weaving_matrix}, $t_1$ is a supermagic labeling on $G_1$ of $(2, 2q)_p$ type for any $p\le q/2$.
Note that $G_2$ is a $2K_{2q[2]}$, which has a supermagic labeling, say $t_2$, by Theorem~\ref{1997_structure_theorem}. For $G_3$, we use any random labeling, which is trivially an $\frac{n}{2}$-almost supermagic labeling. Then as in Remark~\ref{re-dec}, we define the labeling $\tau_{8q}$ on $K_{8q}$ as follows:
for any $e\in E(K_{8q})$,


\begin{equation}\label{primary_construction}
\tau_{8q}(e)=\left\{
    \begin{array}{ll}
    t_1(e)=W_{4q}(i, j), & e=u_iv_j\in E(G_1);\\
    t_2(e)+16q^2, & e\in E(G_2);\\
    32q^2-8q+k, & e \text{ is the $k$th edge in } E(G_3), k\in [4q].
    \end{array}
\right.
\end{equation}
By Remark~\ref{re-dec}, $\tau_{8q}$ is an $\frac{n}{2}$-almost supermagic labeling of  $(2, 2q)_p$ type for any $p\le q/2$. By Corollary~\ref{cor-ml}, we have $r(p, \tau_{8q})\le 0.75$, which means that $\tau_{8q}$ has a better $p$-robustness than all known ones.

To reduce the robustness ratio to $1/2$ and construct $O(n)$-almost supermagic labelings of $K_n$ for all $n$, we will give several recursive constructions in Section~\ref{sec-main_result}. These constructions input an $\alpha$-almost supermagic labeling of $K_n$ of $(m, \ell)_p$ type, and output one of $K_{n'}$ for a large $n'$. We need the output labeling preserves the parameter $\alpha$, and sometimes improves the parameters $m$ and $\ell$ so as to reduce the robustness ratio by Theorem~\ref{cor-ml}.
In order to make recursions more convenient, we need a stronger structure than $\alpha$-almost supermagic labeling, the so-called $b$-astray good labeling, which is defined in the next subsection.

\subsection{$b$-Astray Good Edge Labelings}\label{sec-astray_good}

A $b$-astray good labeling is indeed a candidate of a labeling, where a  small set  $A\subset E(K_n)$ of edges is chosen such that for any $v$, $|D(v)\cap A|\le b$ for some small tolerance $b$. The labels of edges in $A$ can be assigned randomly in  a wide range around the central interval of $[\epsilon ]$, so we call $A$ an ``astray part''.  However, the
labels of $D(v)\backslash A$ are defined very carefully with strong structure to balance the  label sum differences. Here comes the formal definition, which is motivated from ~\cite[Theorem 4]{stewart_1967}.


\begin{definition}\label{definition_3_astray} Let $K_n=(V,E)$, and let $b$ be a nonnegative integer. An edge labeling $t$ of $K_n$ is said to be  \emph{$b$-astray good} if there exists a subset $A\subset E$, which is called the \emph{astray part},  such that the following conditions hold.
\begin{itemize}
\item[(1)] Denote $a=|A|$, then $a\equiv \epsilon\mod 2$ and $S(t, A)=[\frac{\epsilon-a}2+1, \frac{\epsilon+a}2]$, that is, the label set of $A$ is an interval located right in the center of $[\epsilon]$.

\item[(2)] Denote $L$ as the set of edges with labels smaller than those in $A$, that is, $S(t, L)=[\frac{\epsilon-a}2]$. Let $H=E\setminus (A\cup L)$, and hence $S(t, H)=[\frac{\epsilon+a}2+1, \epsilon]$. Call $L$ and $H$ \emph{lower set} and \emph{higher set} of $t$, respectively.

    Then for each $v\in V$, $|D(v)\cap A|\le b$ and $|D(v)\cap L|=|D(v)\cap H|$.
\item[(3)] For each $v\in V$, $s(t, D(v)\backslash A)=|D(v)\backslash A| \cdot (\epsilon+1)/2.$
\end{itemize}
\end{definition}

Definition~\ref{definition_3_astray} (2) ensures that in each $D(v)$, most edges are distributed in the lower set $L$ and the higher set $H$ evenly, and Definition~\ref{definition_3_astray} (3) further requires that the average value of $t$ on those edges is equal to the average value among all edges of $K_n$. Both  conditions contribute to the strong structure of balancing property for a $b$-astray good labeling.

For any two nonnegative integers $b_1\le b_2$, a $b_1$-astray good labeling must be $b_2$-astray good with the same astray part by definition. A lower value of $b$ means a better balancing property. Especially, if $t$ is a $0$-astray good labeling, then the astray part $A$ must be empty, which implies that $t$ must be supermagic by Definition~\ref{definition_3_astray} (3). To guarantee an $\alpha$-almost supermagic labeling with $\alpha=O(n)$, it suffices to require a $b$-astray good labeling with a constant $b$ by the following lemma.

\begin{lemma}\label{4.5n_almost}
If an edge labeling $t$ is $b$-astray good, then $t$ is $\frac {b^2} 2 n$-almost supermagic.
\end{lemma}
\begin{proof} Let $A$ be the astray part of size $a$. Then $|D(v)\cap A|\le b$ for any vertex $v$. So we have $a\le \frac b 2 n$. By Definition~\ref{definition_3_astray} (1), for any $e\in A$, $|t(e)-\frac{\epsilon+1}{2}|\le \frac{a-1}{2}<\frac b 4 n$. For any two vertices $u$, $v$ of $K_n$, let $A_u=A\cap D(u)$ and $A_v=A\cap D(v)$. Then $|A_u|, |A_v|\le b$. By Definition~\ref{definition_3_astray} (3),
\begin{align*}
|s(t, u)-s(t, v)|  & = |s(t, D(u)\backslash A_u)+s(t, A_u)-s(t, D(v)\backslash A_v)-s(t, A_v)|\\
   & = | (n-1-|A_u|) \cdot (\epsilon+1)/2 -(n-1-|A_v|) \cdot (\epsilon+1)/2+s(t, A_u)-s(t, A_v)|\\
   &= |(|A_v|-|A_u|)\cdot (\epsilon+1)/2+s(t, A_u)-s(t, A_v)|\\
   &<(|A_u|+|A_v|)\frac b 4 n\le \frac {b^2} 2 n.
\end{align*}
The last line is due to that  $|t(e)-(\epsilon+1)/2|<\frac b 4 n$ for any $e\in A$.
\end{proof}


The edge labeling $\tau_{8q}$ of $K_{8q}$ in Eq. (\ref{primary_construction}) has been proved to be $4q$-almost supermagic and of $(2, 2q)_p$ type for any $p\le q/2$. However, we are not sure whether it is $b$-astray good for some $b$. Next we modify it to be $1$-astray good with $A=E(G_3)$ being the astray part of size $4q$.


To satisfy Definition~\ref{definition_3_astray} (2), we still use the weaving square to label edges in $G_1$, since each row of the square, which represents the values of $D_{G_1}(v)$ for some $v$, has half of the entries with small values and half with big values.  For edges in $G_2$, we need to construct a special supermagic labeling which has the similar property.
We know that $G_2$ is a disjoint union of two $K_{2q[2]}$'s, say $H$ and $H'$.
 Consider a decomposition of $H$ into two $(2q-1)$-factors $F_1$ and $F_2$. For example, let $F_1$ be the $2K_q$ subgraph of $H$. Viewing $H'$ as a copy of $H$, we have two factors $F'_1$ and $F'_2$ of $H'$.  
 Let $t_0$ be a supermagic labeling on $H$ (and hence also on $H'$), which exists by Theorem~\ref{1997_structure_theorem}. Then the desired labeling $\bar t$ of $G_2$ is defined as follows.
\begin{equation*}
\bar t(e)=\left\{
    \begin{array}{ll}
    t_0(e), & e\in  F_1\cup F'_2;\\
    t_0(e)+|E(H)|, & e\in F_2\cup F'_1.
    \end{array}
\right.
\end{equation*}
It is easy to check that $\bar t$ is a supermagic labeling of $G_2$. For example, for each vertex $v$ of $H$,  we have
\begin{align*}
  s(\bar t, v) & =\sum_{e: ~v\in e\in F_1}t_0(e)+\sum_{e: ~v\in e\in F_2}(t_0(e)+|E(H)|)\\
  & = s(t_0, v)+(2q-1)|E(H)|,
\end{align*}
which is a constant, since $t_0$
 is supermagic.
%

Now a $1$-astray good labeling $T_{8q}$ of $K_{8q}$ is constructed below, where the label set of the astray part $A=E(G_3)$ is located in the center, and labels of edges in $G_1$ and $G_2$ are lifted so that they are all distinct and are distributed evenly to the left and the right of the center.

\begin{construction}\label{const_T8q}
 For any $e\in E(K_{8q})$,
\begin{equation*}
T_{8q}(e)=\left\{
    \begin{array}{ll}
    \bar t(e), & e\in F_1\cup F'_2;\\
    W_{4q}(i, j)+8q^2-4q, & e=u_iv_j\in G_1 \text{ and } W_{4q}(i, j)\le 8q^2;\\
    16q^2-4q+k, & \text{$e$ is the $k$th edge of $G_3$, $k\in[4q]$};\\
    W_{4q}(i, j)+8q^2, & e=u_iv_j\in G_1 \text{ and } W_{4q}(i, j)> 8q^2;\\
    \bar t(e)+16q^2+4q, & e\in F_2\cup F'_1.
\end{array}
\right.
\end{equation*}
\end{construction}

Next we show that $T_{8q}$ is a $1$-astray good labeling and with the same $(m,\ell)_p$ type as $\tau_{8q}$.

\begin{lemma}\label{8q}
  The labeling $T_{8q}$ in Construction~\ref{const_T8q} is a $1$-astray good labeling of $K_{8q}$. Furthermore, $T_{8q}$ is of $(2, 2q)_p$ type for any $p\le q/2$, and hence $r(p,T_{8q})\le 0.75$.
\end{lemma}
\begin{proof}Let $A=E(G_3)$. It is clear that   $T_{8q}$ is a labeling, and $|D(v)\cap A|=1$ for any vertex $v$. Notice that the number of edges of $K_{8q}$ is $\epsilon=32q^2-4q$, and  $S(T_{8q}, A)=[16q^2-4q+1, 16q^2]$, which locates right in the center part of $[\epsilon]$. Then  $L=F_1\cup F'_2\cup \{ u_iv_j\in G_1: M_{4q}(i, j)\le 8q^2\}$ and $H=F_2\cup F'_1 \cup \{ u_iv_j\in G_1: M_{4q}(i, j)> 8q^2\}$. By definitions of $F_i$, $F'_i$, and the weaving square, we know that $|D(v)\cap L|= |D(v)\cap H|$ for each vertex $v$. Finally, for a vertex $v$, say of $H$,

\begin{align*}
  s(T_{8q}, D(v)\setminus A) & =\sum_{e: ~v\in e\in E(G_1)}T_{8q}(e)+\sum_{e: ~v\in e\in F_1}T_{8q}(e)+\sum_{e: ~v\in e\in F_2}T_{8q}(e)\\
  & = \sum_{e: ~v\in e\in E(G_1)}T_{8q}(e)+s(\bar{t},v)+(16q^2+4q)(2q-1).
\end{align*}
The first summation is the row sum of the lifted weaving square, which is $4q(8q^2-2q)+ 2q(1+16q^2)=64q^3-8q^2+2q$, and the second term is label sum of the supermagic $\bar t$, which is $(8q^2+2q)(4q-2)+(2q-1)(4q(4q-2)+1)=64q^3-40q^2+6q-1$, so the total value of  $s(T_{8q}, D(v)\setminus A)$ is $(4q-1)(32q^2-4q+1)$, which equals $(8q-2)\cdot  (\epsilon+1)/2$.

 Furthermore, as $\tau_{8q}$, one can also check that $T_{8q}$ is of $(2, 2q)_p$ type from $W_{4q}$ for any $p\le q/2$, and hence $r(T_{8q}, p)\le 0.75$.
\end{proof}

 In the next section, we will reduce this ratio from $0.75$ to $1/2$ for all $K_n$ by several recursive constructions.

\section{Recursive Constructions and Main Results}\label{sec-main_result}\label{sec:rec}
 This section is devoted to prove our {\bf Main Result}, which is restated below.
 \begin{theorem}\label{theorem_asymptotic_best}
For any real number $\varepsilon>0$ and any large enough $n$, there exists a $ 7n$-almost supermagic labeling  $t$ of $K_n$ such that $r(p, t)\le \frac 1 2+\varepsilon$ for all admissible $p$.
\end{theorem}

The idea of our proof is as follows. Recall that in Lemma~\ref{8q}, the labeling $T_{8q}$ of $K_{8q}$ is $1$-astray good and has $p$-robustness ratio $r(p,T_{8q})\le 0.75$  when $p\le q/2$. Next in Subsection~\ref{sec-ma}, we give constructions of a labeling of $K_{n}$ for all even $n$ from a labeling of $K_{8q}$, then for all odd $n$ from even $n$. All these constructions preserve the astray good property or supermagic property, and keep the same parameters for the $(m, \ell)_p$ type. Since the values of $m$ and $\ell$ control the $p$-robustness ratio by Theorem~\ref{cor-ml}, applying these constructions with $T_{8q}$, we obtain a $7n$-almost supermagic labeling of $K_n$ for all $n$  with $p$-robustness ratio at most $0.75$.

In order to reduce the robustness ratio for all $n$, we need another recursive construction. In Subsection~\ref{sec-reduc}, we give a construction of  labelings of $K_{8q}$ and  $K_{8q+4}$ from one of $K_{4q}$ which preserves the astray good property, and further improves the  parameters for the $(m, \ell)_p$ type. Applying this construction repeatedly, we obtain a  labeling of $K_{8q}$ which has $p$-robustness ratio less than $1/2+\varepsilon$ for any $\varepsilon>0$ when $p\le q/2$. Combining this and Subsection~\ref{sec-ma}, we prove Theorem~\ref{theorem_asymptotic_best} in Subsection~\ref{sec-pr}.

Most constructions work for $b$-astray good labelings  for a general small constant integer $b$. But for our purpose, we can set the largest possible $b$ as $3$, which is the best we can do.
\subsection{From $8q$ to all $n$}\label{sec-ma}

Let $q$ be a positive integer and $t_{(0)}$ be a $3$-astray good edge labeling on $K_{8q}$.
Then we can inductively construct three $3$-astray good edge labelings $t_{(i)}$ on $K_{8q+2i}$ from $t_{(i-1)}$ for all $i\in[3]$. For convenience, the vertices of $K_n$ are denoted by $v_1, v_2, \ldots v_n$ for any $n>0$. So  $K_n$ is viewed as subgraph of $K_{n'}$ if $n'\ge n$. 

\begin{construction}[even case]\label{const_t1}
From a $3$-astray good edge labeling $t_{(0)}$ on $K_{8q}$ with astray part $A_{0}$, lower set $L_0$ and higher set $H_0$, we inductively construct a $3$-astray good labeling $t_{(i)}$  on $K_{8q+2i}$  with some astray part $A_{i}$, lower set $L_i$ and higher set $H_i$, for  $i\in [3]$, as follows.

For $t_{(1)}$,
 let $A_1=A_0\cup\{v_{8q+1}v_{8q+2}\}$. 
 For all $e\in L_0$, $t_{(1)}(e)=t_{(0)}(e)+8q$; for all $e\in H_0$, $t_{(1)}(e)=t_{(0)}(e)+8q+1$.
 For each $i\in[8q]$, let $t_{(1)}(v_{8q+2}v_i)=\epsilon_{8q+2}+1- t_{(1)}(v_{8q+1}v_i)$, where
    \begin{equation*}
    t_{(1)}(v_{8q+1}v_i)=\left\{
    \begin{array}{ll}
    i, & i\in[2q]\cup[6q+1, 8q];\\
   \epsilon_{8q+2}+1-i, & i\in[2q+1, 6q].
    \end{array}
    \right.
    \end{equation*}

For $t_{(2)}$, let $A_2=A_1\cup\{v_{8q+3}v_{8q+4}, v_{8q+4}v_{8q+1}, v_{8q+3}v_{8q+2}, v_{8q+3}v_{8q+1}, v_{8q+4}v_{8q+2}\}$. For all $e\in L_1$, $t_{(2)}(e)=t_{(1)}(e)+8q$; for all $e\in H_1$, $t_{(2)}(e)=t_{(1)}(e)+8q+5$. For each $i\in[8q]$, let $t_{(2)}(v_{8q+4}v_i)=\epsilon_{8q+4}+1-t_{(2)}(v_{8q+3}v_i)$, where
    \begin{equation*}
    t_{(2)}(v_{8q+3}v_i)=\left\{
    \begin{array}{ll}
    i, & i\in[2q]\cup[6q+1, 8q];\\
   \epsilon_{8q+4}+1-i, & i\in[2q+1, 6q].
    \end{array}
    \right.
    \end{equation*}

For $t_{(3)}$, the construction is exactly the same as $t_{(1)}$ but based on $t_{(2)}$.

\end{construction}

\begin{lemma}\label{lemma_stability}
For each $i\in [3]$, the labeling $t_{(i)}$  in Construction~\ref{const_t1}  is a $3$-astray good labeling on $K_{8q+2i}$. Further if $q\ge 2p$ and
 $t_{(0)}$ is of $(m, \ell)_p$ type for some $m$ and $\ell$, then $t_{(i)}$ is of $(m', \ell)_p$ type, where $m'=\max\{m, 3\}$.

\end{lemma}
\begin{proof} The proofs for $t_{(1)}, t_{(2)}$ and $t_{(3)}$ are similar. We only prove $t_{(1)}$ as an example.

First, we show that $t_{(1)}$ is a $3$-astray good labeling.
It is clear that $t_{(1)}$ is injective on $K_{8q+2}\slash A_1$. For any $v_i$, if $i\in [8q]$, then $|S(t_{(1)}, v_i)\cap A_1|= |S(t_{(0)}, v_i)\cap A_0|\le 3$; if $i=8q+1$ or $8q+2$, then $|S(t_{(1)}, v_i)\cap A_1|=1$ by definition. Let
$$
\begin{aligned}L_1=L_0\cup\{v_{8q+1}v_i, i\in[2q]\cup[6q+1, 8q]\}\cup\{v_{8q+2}v_j, j\in [2q+1, 6q]\};\\
H_1=H_0\cup\{v_{8q+1}v_i, i\in[2q+1, 6q]\}\cup\{v_{8q+2}v_j, j\in [2q]\cup[6q+1, 8q]\}.\end{aligned}$$
By definition, it is easy to check that $S(t_{(1)}, L_1)=[8q]\cup(S(t_{(0)}, L_0)+8q)=[(\epsilon_{8q+2}-a_0-1)\slash 2]$, where $a_0=|A_0|$. Similarly, $S(t_{(1)}, H_1)= (S(t_{(0)}, H_0)+8q+1)\cup [\epsilon_{8q+2}-8q+1, \epsilon_{8q+2}]= [(\epsilon_{8q+2}+a_0+1)\slash2+1, \epsilon_{8q+2}].$ So $S(t_{(1)}, A_1)$ locates right in the center of $[\epsilon_{8q+2}]$.

Furthermore, for any $i\in [8q]$,
$$|L_1\cup S(t_{(1)}, v_i)|= |L_0\cup S(t_{(0)}, v_i)|+1= |H_0\cup S(t_{(0)}, v_i)|+1=|H_1\cup S(t_{(1)}, v_i)|;$$
for $i=8q+1$ or $8q+2$, by definition, $|L_1\cup S(t_{(1)}, v_i)|=4q= |H_1\cup S(t_{(1)}, v_i)|.$ So the only thing left is to check the label sum of $D(v_i)\setminus A_1$ for each $i$.  Note that $D(v)$ here is the abbreviation for $D_{K_{8q+2}}(v)$. Since we will deal with the sets $D(v)$ in different complete graphs, we write $D_n(v)$ instead of $D_{K_{n}}(v)$ for simplicity. If $i\in [8q]$,
$$
\begin{aligned}
s(t_{(1)}, D_{8q+2}(v_i)\backslash A_1)&=\sum_{e\in D_{8q+2}(v_i)\cap L_1}t_{(1)}(e)+\sum_{e\in D_{8q+2}(v_i)\cap H_1}t_{(1)}(e)\\
&=\sum_{e\in D_{8q}(v_i)\cap L_0}(t_{(0)}(e)+8q)+ \sum_{e\in D_{8q}(v_i)\cap H_0}(t_{(0)}(e)+8q+1)+ (\epsilon_{8q+2}+1)\\
&=s(t_{(0)}, D_{8q}(v_i)\backslash A_0)+|D_{8q}(v_i)\backslash A_0|(8q+0.5)+ (\epsilon_{8q+2}+1)\\
&=|D_{8q}(v_i)\backslash A_0|\frac{\epsilon_{8q}+1}2+ |D_{8q}(v_i)\backslash A_0|(8q+0.5)+ (\epsilon_{8q+2}+1)\\
&=|D_{8q+2}(v_i)\backslash A_1|\frac{\epsilon_{8q+2}+1}2.
\end{aligned}$$
If $i=8q+1$ or $8q+2$, then $|D_{8q+2}(v_i)\backslash A_1|=8q$, and $s(t_{(1)}, D_{8q+2}(v_i)\backslash A_1)= \frac{(8q+1)}2\cdot 4q+\frac{(\epsilon_{8q+2}+ \epsilon_{8q+2}-8q+1)}2\cdot 4q= \frac{\epsilon_{8q+2}+1}2 \cdot 8q.$ So $t_{(1)}$ is 3-astray good.

Next we  prove that $t_{(1)}$ is of $(m', \ell)_p$ type. For each $i\in [8q]$, the fact that $S(t_{(0)}, v_i)$ is of $(m, \ell)_p$ type implies that $S(t_{(1)}, v_i)$ is also of $(m, \ell)_p$ type. For $i= 8q+1$ or $8q+2$, $S(t_{(1)}, v_i)$ is of $(3, 8q)_p$ type by definition. Since $\ell\leq 8q$, $t_{(1)}$ is of $(m', \ell)_p$ type with $m'=\max\{m, 3\}$.
\end{proof}


Next, we show how to deal with the cases for odd $n$. The following construction from $K_{2k}$ to $K_{2k+1}$ preserves the  $(m, \ell)_p$ type but may not have the $3$-astray good property.

\begin{construction}[odd case]\label{const_odd}
Let $t$ be a $3$-astray good edge labeling on $K_{2k}$ with astray part $A$ of size $a$, lower set $L$ and higher set $H$, such that $|L|=|H|=l$ and $2l+a=\epsilon_{2k}$. Denote $n=2k+1$. We  construct a labeling $t'$ on $K_{n}$ as follows. For any $e\in K_{2k}$,
\begin{equation*}
t'(e)=\left\{
    \begin{array}{ll}
    t(e), & e\in L;\\
    t(e)+k, & e\in A;\\
    t(e)+2k, & e\in H.
    \end{array}
\right.
\end{equation*}
For all $i\in[2k]$,
\begin{equation*}
t'(v_{n}v_i)=\left\{
    \begin{array}{ll}
    l+a+k+i, & i\in[k];\\
    l+i-k, & i\in [k+1, 2k].
    \end{array}
\right.
\end{equation*}
\end{construction}

\begin{lemma}\label{lemma_for_odd_t} Let $n=2k+1$.
The labeling $t'$ on $K_n$ in Construction~\ref{const_odd} is $7n$-almost supermagic.
Further if $n\ge 4p+1$ and $t$ is of $(m, \ell)_p$ type for some $m$ and $\ell$, then $t'$ is of $(m', \ell)_p$ type for $m'=\max\{m, 2\}$.
\end{lemma}
\begin{proof}
 It is easy to check that $t'$ is an edge labeling. The two disjoint $1$-APs of length $k$ ($\geq 2p$) in $S(t', v_{n})$ imply the correctness of the second statement. So we only need to prove that $t'$ is $7n$-almost supermagic, that is, $|s(t', v)-s(t', u)|\leq 7n$ for all $u\neq v$.

Note that $\epsilon_n=\epsilon_{n-1}+2k=\epsilon_{n-1}+2l+a=2k^2+k$. Then $s(t', v_{n})=k(1+\epsilon_n)$. For each $i\in [2k]$, $|D(v_i)\cap A|=1$ or $3$ since $t$ is $3$-astray good. Then $|D(v_i)\cap H|=k-1$ or $k-2$, respectively. For both cases,  $s(t', v_i)=t'(v_{n}v_i)+ s(t, v_i)+ k(2k-1)=t'(v_{n}v_i)+ s(t, v_i)+ \epsilon_{n-1}$.

By Lemma~\ref{4.5n_almost}, $t$ is a $9k$-almost supermagic of $K_{n-1}$. So the difference between $s(t, v_i)$ and the average label sum in $K_{n-1}$ is at most $9k$ for $i\in[2k]$, that is, $|s(t, v_i)-(2k-1)(\frac{1+\epsilon_{n-1}} 2)|\le 9k$. 
On the other hand, $a=|A|\le 2k\times\frac 3 2=3k$. So for any $i\in[2k]$, $|t'(v_{n}v_i)-\frac{1+\epsilon_{n}}2|\le (a+2k)\slash 2\le 2.5k$ by the construction of $t'(v_{n}v_i)$.
Then for any $i\in[2k]$,
$$
\begin{aligned}
|s(t', v_i)-s(t', v_{n})|& =|t'(v_{n}v_i)+s(t, v_i)+ \epsilon_{n-1}-k(1+\epsilon_n)|\\
 & \le \left|t'(v_{n}v_i)-\frac{1+\epsilon_{n}}2\right|+ \left|s(t, v_i)-(2k-1)\left(\frac{1+\epsilon_{n-1}} 2\right)\right|\\
 & \le 2.5 k +9k< 7(2k+1).
\end{aligned}
$$
For $i\neq j$ in $[2k]$,
$$
\begin{aligned}
|s(t', v_i)-s(t', v_j)|& \le |t'(v_{n}v_i)-t'(v_{n}v_j)|+|s(t, v_i)-s(t, v_j)|\\
 & \le a+2k+9k\\
 & \le 14k< 7(2k+1).
\end{aligned}
$$
To sum up, $t'$ is $7n$-almost supermagic.
\end{proof}

We remark that the labeling $t'$ is not $3$-astray good anymore for any astray part $A'$ with $A\subset A'$. That is because the set $S(t',v_n)$ contains too many values near to the center of $[\epsilon_n]$. The argument is as follows. Since $t$ is $3$-astray good and $2k$ is even, for each $v_i, i\in [2k]$, $|D(v_i)\cap A|$ is odd. Here $D(v_i)$ is always the edge set under $K_n$. If on the contrary that $t'$ is also $3$-astray good, but now $n=2k+1$ is odd, then $|D(v_i)\cap A'|$ is even. Hence there exists a new edge $e\in D(v_i)$ such that $e\in A'\setminus A$ for any $i\in [2k]$. Notice that the labels of $A'$ under $t'$ form an interval centered in $[\epsilon_n]$, and the edge in $D(v_i)\backslash A$ which has the value nearest to the center is $v_iv_n$. So $v_iv_n\in A'$ for any $i\in [2k]$. But then $|S(t',v_n)\cap A'|\geq 2k$, which contradicts to Definition~\ref{definition_3_astray}(2).

\subsection{Improving the $(m, \ell)_p$ type}\label{sec-reduc}
The following construction is a key ingredient in the proof of our {\bf Main Result}, which improves the $(m, \ell)_p$ type for the resultant labeling. It generalizes the idea in Construction~\ref{const_T8q} for $T_{8q}$.
\begin{construction}[$4q$ to $8q$]\label{const_t0}
Denote the vertices in $K_{4q}$ as $v_i, i\in [4q]$, and let $t$ be a $3$-astray good labeling on $K_{4q}$ with astray part $A$ of size $a$, and lower set $L$ and higher set $H$, such that $|L|=|H|=l$ and $2l+a=\epsilon_{4q}$.

Now we extend $K_{4q}$ to $K_{8q}$ by adding $4q$ new vertices $u_i, i\in[4q]$.  Let $A'=A\cup \{u_iu_j: v_iv_j\in A\}$. For any edge $e\in E(K_{8q})$ with $e\notin A'$, define $t'(e)\in [\epsilon_{8q}]$ as follows.
\begin{equation*}
t'(e)=\left\{
    \begin{array}{ll}
    t(e), &e\in L;\\
    t(v_iv_j)+l, &e=u_iu_j\text{ s.t. }v_iv_j\in L;\\
    W_{4q}(i, j)+2l, &e=u_iv_j\text{ s.t. }W_{4q}(i, j)\le 8q^2;\\
W_{4q}(i, j)+2l+2a, &e=u_iv_j\text{ s.t. }W_{4q}(i, j)> 8q^2;\\
    t(v_iv_j)+\Delta_\epsilon-l, & e=u_iu_j\text{ s.t. }v_iv_j\in H;\\
t(e)+\Delta_\epsilon, &e\in H.
\end{array}
\right.
\end{equation*}
Here $\Delta_\epsilon=|E(K_{8q})|-|E(K_{4q})|=\epsilon_{8q}-\epsilon_{4q}$. It is easy to check that $t'$ is a bijection from $E(K_{8q})\setminus A'$ to $[\epsilon_{8q}/2-a]\cup [\epsilon_{8q}/2+a+1, \epsilon_{8q}]$. So we can complete $t'$ to be a labeling by assigning edges  in $A'$ with  labels in the remaining set $[\epsilon_{8q}/2-a+1, \epsilon_{8q}/2+a]$.
\end{construction}

\begin{lemma}\label{lemma_from_4q_to_8q} The labeling $t'$ in Construction~\ref{const_t0}  is a $3$-astray good labeling on $K_{8q}$. Further if $q\ge 2p$ and
 $t$ is of $(m, \ell)_p$ type for some $m$ and $\ell$, then $t'$ is of $(m+2, \ell+2q)_p$ type.
\end{lemma}
\begin{proof}  Since $K_{4q}$ is a subgraph of $K_{8q}$ in our construction, we use $D_{4q}(v)$ and $D_{8q}(v)$ to distinguish the sets in $K_{4q}$ and $K_{8q}$, following the notations in the proof of Lemma~\ref{lemma_stability}.  Since
 $t$ is a $3$-astray good labeling on $K_{4q}$ with astray part $A$, lower set $L$ and higher set $H$,  then for each $v_i$, $|D_{4q}(v_i)\cap A|\leq 3$,  $|D_{4q}(v_i)\cap L|=|D_{4q}(v_i)\cap H|$, and $s(t, D_{4q}(v_i)\backslash A)=|D_{4q}(v_i)\backslash A| \cdot (\epsilon_{4q}+1)/2.$

 Next, we show that $t'$ is a $3$-astray good labeling on $K_{8q}$. It is clear that $S(t',A')=[\epsilon_{8q}/2-a+1, \epsilon_{8q}/2+a]$. By definition, the lower set and the higher set for $t'$ are as follows.
$$L'=L\cup \{u_iu_j: v_iv_j\in L\}\cup\{u_iv_j: W_{4q}(i, j)\le 8q^2\},$$
$$H'=H\cup \{u_iu_j: v_iv_j\in H\}\cup\{u_iv_j: W_{4q}(i, j)> 8q^2\}.$$
Since $|L|=|H|$, combining the properties of $W_{4q}$,  we have $|L'|=|H'|$.

{
Then for each vertex $v_j$,  $j\in [4q]$,  $|A'\cap D_{8q}(v_j)|=|A\cap D_{4q}(v_j)|\le 3$, \[|L'\cap D_{8q}(v_j)|=|L\cap D_{4q}(v_j)|+|\{u_iv_j:  W_{4q}(i, j)\le 8q^2\}|=|L\cap D_{4q}(v_j)|+2q,\] and similarly, $|H'\cap D_{8q}(v_j)|=|H\cap D_{4q}(v_j)|+2q$. Since $|L\cap D_{4q}(v_j)|=|H\cap D_{4q}(v_j)|$, we have $|L'\cap D_{8q}(v_j)|=|H'\cap D_{8q}(v_j)|$. For vertex $u_i$, the same argument shows $|L'\cap D_{8q}(u_i)|=|H'\cap D_{8q}(u_i)|$.

Finally, we need to show that for each vertex $v$,  $s(t', D_{8q}(v)\setminus A')=|D_{8q}(v)\setminus A'| \cdot (\epsilon_{8q}+1)/2.$ By symmetry, we only prove it for $v_j$, $j\in [4q]$. It is easy to see that $D_{8q}(v_j)\setminus A'=(D_{4q}(v_j)\setminus A)\cup \{u_iv_j:i\in [4q]\}$, and $(D_{4q}(v_j)\setminus A)$ is halved into $D_{4q}(v_j)\cap L$ and $D_{4q}(v_j)\cap H$. Then
\begin{align*}
s(t', D_{8q}(v_j)\setminus A')=&\sum_{e\in D_{4q}(v_j)\setminus A} t'(e)+\sum_{i=1}^{4q} t'(u_{i}v_j)\\
=&\sum_{e\in D_{4q}(v_j)\cap L} t(e)+\sum_{e\in D_{4q}(v_j)\cap H} (t(e)+\Delta_\epsilon)+\\&\sum_{i:W_{4q}(i, j)\le 8q^2} (W_{4q}(i, j)+2l)+\sum_{i:W_{4q}(i, j)> 8q^2} (W_{4q}(i, j)+2l+2a)\\
=& \sum_{e\in D_{4q}(v_j)\setminus A} t(e)+ \Delta_\epsilon  |D_{4q}(v_j)\cap H| + \sum_{i=1}^{4q}W_{4q}(i, j)+ 8ql+4qa.
\end{align*}
Since $\sum_{e\in D_{4q}(v_j)\setminus A} t(e)=s(t, D_{4q}(v_j)\backslash A)=|D_{4q}(v_j)\backslash A| \cdot (\epsilon_{4q}+1)/2$, we have $\sum_{e\in D_{4q}(v_j)\setminus A} t(e)+ \Delta_\epsilon  |D_{4q}(v_j)\cap H|=|D_{4q}(v_j)\setminus A|\cdot (\epsilon_{8q}+1)/2$. Since
 $\sum_{i=1}^{4q}W_{4q}(i, j)=32q^3 +2q$, and $2l+a=\epsilon_{4q}=8q^2-2q$, we have $\sum_{i=1}^{4q}W_{4q}(i, j)+ 8ql+4qa=2q\cdot (\epsilon_{8q}+1)$. So $s(t', D_{8q}(v_j)\setminus A')=|D'(v_j)\setminus A'|\cdot (\epsilon_{8q}+1)/2$.

  Combining all, we have shown that $t'$ is a $3$-astray good labeling on $K_{8q}$. Besides, we see that the weaving square in Construction~\ref{const_t0} introduces two $1$-APs of length $q$ for each $D_{8q}(v)\backslash D_{4q}(v)$, so we have the second statement.} \end{proof}
Combining Lemma~\ref{lemma_from_4q_to_8q} and Lemma~\ref{lemma_stability} for the result of $t_{(2)}$, we have the following corollary.

\begin{corollary}\label{coro_from_4q_to_8q}
Let $p$, $q\ge 2p$, $m$ and $\ell$ be positive integers. If there exists a $3$-astray good labeling of $(m, \ell)_p$ type on $K_{4q}$,  then there exists a $3$-astray good labeling of $(m+2, \ell+2q)_p$ type on $K_{8q}$, and  a $3$-astray good labeling of $(m+2, \ell+2q)_p$ type on $K_{8q+4}$.
\end{corollary}

\subsection{Proof of Theorem~\ref{theorem_asymptotic_best}}\label{sec-pr}

%
%

\begin{proof}
Let $s$ be the smallest integer such that $\frac 1 {2^s}<\varepsilon$. Define a non-increasing sequence $n_0, n_1, \ldots, n_s$ of positive integers as follows. Let $n_0=n$, $n_1=8\lfloor\frac n 8\rfloor$, and for any $ 2\le i\le s$, $n_{i}=4\lfloor\frac {n_{i-1}} 8\rfloor$. By the definition, $n_i, i\ge 2$ are all multiples of $4$. Since $n$ is large enough and $p=o(n)$, we can assume that $n_s\ge 8p+4$.

Our strategy starts from a good labeling $t_s$ of $K_{n_s}$, then applies Constructions~\ref{const_t1}-\ref{const_t0}  recursively, to construct a labeling $t_{i-1}$ of  $K_{n_{i-1}}$ from $t_i$ of $K_{n_{i}}$, for all $i\in [s]$. In each step, the resultant labeling has a better $(m, \ell)_p$ type, thus reduces the $p$-robustness ratio step by step, until the final labeling $t=t_0$ of $K_{n_0}=K_n$ has a ratio approaching $1/2$.

If $n_s\equiv 0\mod 8$, let $t_s=T_{n_s}$ in Construction~\ref{const_T8q}. If $n_s\equiv 4\mod 8$, then $n_s-4\equiv 0\mod 8$. Let $t_s$ be the output $t_{(2)}$ in Construction~\ref{const_t1} by inputting $t_{(0)}=T_{n_s-4}$ in Construction~\ref{const_T8q}. For both cases, $t_{s}$ is $3$-astray good and of $(m_s, \ell_s)_p$ type on $K_{n_s}$ with $m_s=3$ and $\ell_s=\frac{n_s-4}{4}$.


For $2\le i\le s$, suppose that we have got a $3$-astray good labeling of $(m_i, \ell_i)_p$ type on $K_{n_i}$. Since $n_{i-1}=2n_i$ or $2n_i+4$, by Corollary~\ref{coro_from_4q_to_8q}, there exists a $3$-astray good labeling $t_{i-1}$ on $K_{n_{i-1}}$ of $(m_{i-1}, \ell_{i-1})_p$ type such that $m_{i-1}\le m_i+2$ and $\ell_{i-1}\ge \ell_i+\frac{n_i}2$.
Especially when $i=2$, there exists a $3$-astray good labeling $t_1$ on $K_{n_1}$ of $(m_1, \ell_1)_p$ type with $m_1\le 2s+1$ and $\ell_1\ge \frac{n_s-4}4+\frac{n_s}2+\cdots+\frac{n_2}2$. Since
 $\frac{n_{i-1}}2 \ge n_i\ge \frac {n_{i-1}} 2 -2$ for any $i\in [2, s]$, we have ${ \ell_1\ge (1-\frac 1 {2^s})n_2-2s}$. Then applying Construction~\ref{const_t1} for even $n$, or further Construction~\ref{const_odd} for odd $n$, we get a $7n$-almost supermagic labeling $t=t_0$ of $(m_1, \ell_1)_p$ type on $K_n$ from $t_1$ on $K_{n_1}$.


It is left to compute the $p$-robustness ratio of $t$.
By Lemma~\ref{1-AP_dominent_lemma}, for any $p$-swap $\theta$ and any two different vertices $u, v\in V(K_n)$, $|s(\theta t, u)-s(\theta t, v)|\le 7n+ 2m_1p^2+2p(n-\ell_1-1)$. When $n$ goes to infinity and $p=o(n)$, applying $m_1\le 2s+1$, $\ell_1\ge (1-\frac 1 {2^s})n_2-2s$ and $n_2\ge\frac{n-7} 2$ into the inequality, we have $|s(\theta t, u)-s(\theta t, v)|=(1+\frac 1 {2^s}) pn+o(pn)$. Then $r(t, n)\le \frac 1 2+\frac 1 {2^{s+1}}<\frac 1 2+\epsilon$.
\end{proof}


Theorem~\ref{theorem_asymptotic_best} shows that there exists a $7n$-almost supermagic labeling of $K_n$ whose robustness ratio is close to $1/2$.  Another direction is that, if we weaken the ratio, for example, to at most $\frac 3 4$, then we can pursue a stronger balancing property for all large enough $n$. In fact, by applying Construction~\ref{const_t1} with $T_{8q}$ directly and carefully labeling the edges in the astray parts, we can get  an $\frac n 2$-almost supermagic labeling $T_n$ of $K_n$ for all even $n$, and then an $n$-almost supermagic labeling for all odd $n$ by Construction~\ref{const_odd}, whose robustness ratio is $\frac 3 4$. The labeling plans for the astray parts $A_i$ in $T_{8q+2i}$  of $K_{8q+2i}$, $i\in [3]$, are as follows.


\begin{itemize}
\item[(1)] In $T_{8q+2}$, set $T_{8q+2}(v_{2k-1}v_{2k})=16q^2+4q+k$  for all $k\in[4q+1]$.
\item[(2)] In $T_{8q+4}$, set $T_{8q+4}(v_{2k-1}v_{2k})=16q^2+12q+k$ for all $k\in[2q+1]$; set $T_{8q+4}(v_{2k-1}v_{2k})=16q^2+12q+4+k$ for all $k\in[2q+2, 4q+2]$. Besides, the values of $T_{8q+4}$ on $v_{8q+4}v_{8q+1}$, $v_{8q+3}v_{8q+2}$, $v_{8q+3}v_{8q+1}$ and $v_{8q+4}v_{8q+2}$ are $16q^2+14q+2$, $16q^2+14q+3$, $16q^2+14q+4$ and $16q^2+14q+5$, respectively.
\item[(3)] The values of $T_{8q+6}$ on $v_{8q+4}v_{8q+1}$, $v_{8q+3}v_{8q+2}$, $v_{8q+3}v_{8q+1}$ and $v_{8q+4}v_{8q+2}$ are $16q^2+22q+6$, $16q^2+22q+7$, $16q^2+22q+9$ and $16q^2+22q+10$, respectively. Besides,
\begin{equation*}
T_{8q+6}(v_{2k-1}v_{2k})=\left\{
    \begin{array}{ll}
    16q^2+20q+4+k, & k\in[2q+1],\\
   16q^2+22q+8, & k=2q+2,\\
   16q^2+20q+8+k, & k\in[2q+3, 4q+3].
    \end{array}
\right.
\end{equation*}
\end{itemize}

\section{Conclusion}\label{conc.}
Motivated by the access balancing problem in dynamical distributed storage system, we study the robustness of almost supermagic edge labelings of  graphs under limited-magnitude swaps.
In particular, we construct $O(n)$-almost supermagic labelings on complete graph $K_n$ that are asymptotically optimal in terms of the $p$-robustness ratio for  any large $n$. There are several interesting problems which deserve further study for complete graphs.
\begin{itemize}
\item[(1)] Find constructions of $O(n)$-almost supermagic labelings on $K_n$ whose $p$-robustness ratios converge to $\frac 1 2$ faster than ours for all admissible $p$.
\item[(2)] Find constructions of $o(n)$-almost supermagic labelings on $K_n$ with small $p$-robustness ratio.

\end{itemize}

 Since optimal FR codes can also be constructed from  Tur\'{a}n graphs, it is interesting to study the same problem for  Tur\'{a}n graphs. Similar estimate can be made on the robustness as in Theorem~\ref{bou}, which helps to define the robustness ratio. It is worth to mention again that in the view of labeling, the weaving square is indeed a supermagic labeling on the Tur\'{a}n graph $T(2, 8q)$ (see Eq.~(\ref{primary_construction})) and can be proved to be optimal in terms of robustness by simple arguments. Our method of using weaving squares can also be extended to construct such optimal labelings for any Tur\'{a}n graph $T(r,n)$ with $4r\mid n$.  
 Nevertheless, for the whole problem, we still have a lot of unexplored areas. We list some of them as follows.
 \begin{itemize}
 \item[(3)] Solve the problem for Tur\'{a}n graph $T(r, n)$ for any $r\mid n$ .
 \item[(4)] Consider the graph $K_n$ but with a different popularity labeling system. For example, the labels are directly proportional to the popularities  obeying the Zipf law \cite{breslau1999web}.
 \item[(5)] Consider other (hyper)graphs which deduce optimal FR codes. For example,  transversal designs and generalized polygons can produce
optimal FR codes whose replication numbers go larger than $2$ \cite{Turan_Ski2015}.
 \end{itemize}

\appendices
\section{Old constructions are not robust enough}\label{sec-old_constructions}
In this appendix, we show that factorial constructions in  \cite{Colbourn2021} and the inductive constructions mainly in \cite{stewart_1967} of (almost) supermagic labelings
 on $K_n$ are not robust enough under $p$-swaps by applying Lemmas~\ref{slice_lemma}, ~\ref{vertex_link_lemma} and Corollary~\ref{link_corollary}.

\subsection{The Factorial Constructions are Not Good Enough}

In \cite{Colbourn2021} factor decompositions are used to construct supermagic or 1-almost supermagic labelings for $K_n$ when $n\equiv 0, 2, 3\mod 4$. We denote such constructions as $t_f$. Some facts for $t_f$ are listed below.

\begin{fact}\label{fact_for_factor_cons}
\begin{itemize}
\item[(1)] When $n\equiv 0, 2\mod 4$ {(see Lemma 7 and Lemma 3 in~\cite{Colbourn2021})}, there exists a $1$-factor decomposition $\{F_1, F_2, \ldots, F_{n-1}\}$ of $K_n$ such that for any $i\in [n-1]$, $S(t_f, F_i)=[(i-1)\frac n 2+1, i\frac n 2]$.
\item[(2)] When $n\equiv 3 \mod 4$ {(see Lemma 4 in~\cite{Colbourn2021})}, there exists a $2$-factor decomposition $\{F_1, F_2, \ldots, F_{\frac{n-1}2}\}$ of $K_n$ such that for any $i\in [\frac{n-1}2]$, $S(t_f, F_i)=[(i-1)n+1, in]$.
\end{itemize}
\end{fact}

When $n\equiv 0, 2\mod 4$, for construction $t_f$, $[\epsilon]$ is decomposed into $n-1$ pieces of $1$-APs $I_i=S(t_f, F_i)=[(i-1)\frac n 2+1, i\frac n 2]$ for all $i\in [n-1]$. Since each $F_i$ is a $1$-factor, for any vertex $v$ there exists exactly one edge in $D(v)\cap F_i$. So by Lemma~\ref{slice_lemma}, $r(p, t_f)=1$ for any feasible $p$.

When $n\equiv 3\mod 4$, for construction $t_f$ and any vertex $v$, consider disjoint $1$-APs $I_i=S(t_f, F_i)=[(i-1)n+1, in]$ for all $i\in[\frac{n-1}2]$ forming a decomposition of $[\epsilon]$. The length of any $I_i$ is $n>4p$. For any $i\in\frac{n-1}2$, $|S(t_f, v)\cap I_i|=2$. So by Corollary~\ref{link_corollary}, we can choose $h$ and $p^*$ near $p$ such that $r(p, t_f)\ge r(p^*, t_f)$ can be as near to $1$ as possible.

\subsection{The Inductive Constructions are Not Good Enough}

The inductive constructions use known supermagic labelings on small complete graphs to construct new supermagic labelings on a larger graph. We consider three additive constructions for $n\equiv 1, 2, 3\mod 4$, respectively in \cite{stewart_1967}, and one multiplicative construction for $n\equiv 1\mod 4$ in \cite{Colbourn2021}.

In the former three cases, a base labeling on $K_n$ is used to construct a new labeling on $K_{n+4}$. Here we list some facts we need in the following proofs. For more details one may refer to the original paper.

\begin{fact}\label{fact_for_add_induction}
We denote the base supermagic labeling on $K_n$ as $t_B$, and the new supermagic labeling on $K_{n+4}$ as $t_N$. The vertices of $K_i$, $i=n$ or $n+4$ are denoted as $v_1, v_2, \ldots, v_i$. 
\begin{itemize}
\item[(1)] When $n=4k+3$ {(see Theorem 3, case 1 in~\cite{stewart_1967})}, for all edges $e\in K_n$, $t_N(e)=t_B(e)+8k+9$. Furthermore, for any $v_i, i\in[n]$, in $\{t_N(v_iv_j): j\in [n+1, n+4]\}$ there are exactly two members in $[8k+9]$ and two members in $[\epsilon_{n+4}-8k-8, \epsilon_{n+4}]$.
\item[(2)] When $n=4k+1$ {(see Theorem 3, case 2 in~\cite{stewart_1967})}, for all edges $e\in K_n$, $t_N(e)=t_B(e)+8k+5$. Furthermore, for any $v_i, i\in[n]$, in $\{t_N(v_iv_j): j\in [n+1, n+4]\}$ there are exactly two members in $[8k+5]$ and two members in $[\epsilon_{n+4}-8k-4, \epsilon_{n+4}]$.
\item[(3)] When $n=4k+2$ {(see Theorem 4 in~\cite{stewart_1967})}, the induction needs the base case $t_B$ to satisfy all requirements for being $1$-astray good (see Definition~\ref{definition_3_astray}) except the third one, i.e., we do not require $s(t_B, D(v)\backslash A)=|D(v)\backslash A| \cdot (\epsilon_n+1)/2$ for any vertex $v$. We call a labeling satisfying this weak version of $1$-astray goodness as {\emph{a weak $1$-astray good labeling}}. Then under this definition, the output
    $t_N$ is also weak $1$-astray good. For any $e$ such that $t_B(e)\le 2k(2k+1)$, $t_N(e)=t_B(e)+8k+6$; for any $e$ such that $t_B(e)> (2k+1)^2$, $t_N(e)=t_B(e)+8k+8$.
    For any $v_i, i\in[n]$, in $\{t_N(v_iv_j): j\in [n+1, n+4]\}$ there are exactly two members in $[8k+6]$ and two members in $[\epsilon_{n+4}-8k-5, \epsilon_{n+4}]$.
\end{itemize}
\end{fact}

In \cite{Colbourn2021} for $n\equiv 1\mod 4$, an inductive construction using supermagic labeling on $K_{2s+1}$ for some even $s$ to construct supermagic labeling on $K_{4s+1}$ is given {(see Lemma 6 in~\cite{Colbourn2021})}, which  satisfies the following fact.
\begin{fact}\label{fact_for_multi_ind_cons}
The new construction and the base construction are defined as $t_N$ and $t_B$. Suppose the vertex set of $K_{2s+1}$ is $\{v_i: i\in[2s]\}\cup\{v_\infty\}$, and the vertex set of $K_{4s+1}$ is $\{(v_i, j): i\in [2s], j\in[2]\}\cup\{v_\infty\}$. For any vertex $v\ne v_\infty$ and $i\in[0, s-1]$, $|S(t_N, v)\cap[2si+1, 2s(i+1)]|=1$ and $|S(t_N, v)\cap[6s^2+2s(i+1)+1, 6s^2+2s(i+2)]|=1$. Furthermore, if $v=(v_i, j)$ for some $i\in [2s]$ and $j\in[2]$, and there exists an interval $[x, y]\subset[\epsilon_{2s+1}]$ such that $|S(t_B, v_i)\cap [x, y]|=r$ for some integer $r$, then $|S(t_N, v)\cap[2s^2+2x-1, 2s^2+2y]|=r$. All edges containing $v_\infty$ in $K_{4s+1}$ have $t_N$ values in $[2s^2+1, 6s^2+2s]$.
\end{fact}

The facts above can derive the following proposition. From now on, we also use the labels of the facts for the corresponding constructions they refer to. For example, the additive inductive construction when $n\equiv 2\mod 4$ can be denoted by Fact~\ref{fact_for_add_induction} (3).

\begin{proposition}\label{propositioin_large_case}
Consider any supermagic labeling $t$ on $K_n$ constructed from a base supermagic labeling $t_B$ on $K_N$ by iterating inductive constructions from above.
\begin{itemize}
\item[(1)] If $n\equiv 2 \text{ (or $3$)}\mod 4$, the labeling $t$ is obtained from $t_B$ by using $\frac{n-N}{4}$ times of construction Fact~\ref{fact_for_add_induction} (3) (or Fact~\ref{fact_for_add_induction} (1), respectively). Let $\ell, m$ and $\gamma$ be nonnegative integers. If there exist $V_N\subset V(K_N)$ and $m$ disjoint $1$-APs $I_i\subset [\epsilon_N], i\in[m]$ with lengths no smaller than $\ell$, such that for any $v\in V_N$, $|S(t_B, v)\backslash (\cup_{i=1}^m I_i)|\le \gamma$ and $|S(t_B, v)\cap I_i|\le 2$ for any $i\in [m]$, then there exist $m'$ disjoint $1$-APs $I_j'\subset [\epsilon_n], j\in[m']$ with $m'=m+\frac{n-N}2$ and lengths no smaller than $\min\{\ell, 2N+2\}$, such that for any $v\in V_N$, $|S(t, v)\backslash (\cup_{j=1}^{m'} I_j')|\le \gamma$ and $|S(t, v)\cap I_j'|\le 2$ for any $j\in [m']$.

\item[(2)] If $n\equiv 1\mod 4$, then there exist some nonnegative integers $a$ and $b$ such that $t$ is obtained from $t_B$ by using in total $a$ times of construction Fact~\ref{fact_for_add_induction} (2) and $b$ times of construction Fact~\ref{fact_for_multi_ind_cons}.
    Let $\ell, m$ and $\gamma$ be nonnegative integers. If there exist $V_N\subset V(K_N)$ and $m$ disjoint $1$-APs $I_i\subset [\epsilon_N], i\in[m]$ with lengths no smaller than $\ell$, such that for any $v\in V_N$, $|S(t_B, v)\backslash (\cup_{i=1}^m I_i)|\le \gamma$ and $|S(t_B, v)\cap I_i|\le 2$ for any $i\in [m]$, then there exist some $m'\le m+2a+nb$, a vertex set $V_n\subset V(K_n)$ with $|V_n|\ge 2^b(|V_N|-b)$, and $m'$ disjoint $1$-APs $I_j'\subset [\epsilon_n], j\in[m']$ with lengths no smaller than $\min\{\ell, N-1\}$, such that for any $v\in V_n$, $|S(t, v)\backslash (\cup_{j=1}^{m'} I_j')|\le \gamma$ and $|S(t, v)\cap I_j'|\le 2$ for any $j\in [m']$.
\end{itemize}
\end{proposition}
\begin{proof}
The first case can be easily derived by the same proving process of the second one, so we only give the proof when $n\equiv 1\mod 4$.

There exists a tower of induction relationship with $a+b$ steps from $t_B$ to $t$, say $t_B=t_0\to t_1\to t_2\cdots\to t_{a+b}=t$, such that $t_i$ is a supermagic labeling on $K_{N_i}$ derived from $t_{i-1}$ by one of the two constructions, Fact~\ref{fact_for_add_induction} (2) or Fact~\ref{fact_for_multi_ind_cons},      for any $i\in[a+b]$. So $N_0=N$, $N_{a+b}=n$, and $N_i=N_{i-1}+4$ or $2N_{i-1}-1$.

For any $i\in[0, a+b-1]$, if there exist disjoint $1$-APs $I_1^{(i)}, I_2^{(i)}, \ldots, I_{m_i}^{(i)}\subset [\epsilon_{N_i}]$ with the smallest length $\ell_i$ and a vertex set $V_{N_i}\subset V(K_{N_i})$ such that for any $v\in V_{N_i}$, $|S(t_i, v)\cap I_j^{(i)}|\le 2$ for any $j\in[m_i]$ and $(N_i-1)-|S(t_i, v)\cap(\cup_{j=1}^{m_i}I_j^{(i)})|\le \gamma$ for some  $\gamma$, then $t_{i+1}$ will follow one of the following two cases.

If $t_i\to t_{i+1}$ is by construction Fact~\ref{fact_for_add_induction} (2), $N_{i+1}=N_i+4$. Let $k=(N_i-1)\slash 4$ and define $m_{i+1}=m_i+2$. Consider $m_{i+1}$ disjoint $1$-APs $I_1^{(i+1)}, \ldots, I_{m_{i+1}}^{(i+1)}$ from $[\epsilon_{N_{i+1}}]$ as follows,
\begin{equation*}
I_j^{(i+1)}=
\begin{cases}
I_j^{(i)}+8k+5, &\text{ for all }j\in[m_i];\\
[8k+5], & \text{ when }j=m_i+1;\\
[\epsilon_{4k+5}-8k-4, \epsilon_{4k+5}], & \text{ when }j=m_i+2.
\end{cases}
\end{equation*}
So the smallest length $\ell_{i+1}=\min\{\ell_i, 8k+5\}$. Let $V_{N_{i+1}}=V_{N_i}$. By Fact~\ref{fact_for_add_induction} (2), any vertex $v\in V_{N_{i+1}}$ satisfies $|S(t_{i+1}, v)\cap I_j^{(i+1)}|\le 2$ for any $j\in[m_{i+1}]$ and $(N_{i+1}-1)-|S(t_{i+1}, v)\cap(\cup_{j=1}^{m_{i+1}}I_j^{(i+1)})|\le \gamma$.

If $t_i\to t_{i+1}$ is by construction Fact~\ref{fact_for_multi_ind_cons}, $N_{i+1}=2N_i-1$. Let $s=(N_i-1)\slash2$ and define $m_{i+1}=m_i+2s$. Consider $m_{i+1}$ disjoint $1$-APs $I_1^{(i+1)}, \ldots, I_{m_{i+1}}^{(i+1)}$ from $[\epsilon_{N_{i+1}}]$ as follows. For any $j\in[m_i]$, if $I_j^{(i)}=[x, y]$ for some integers $x$ and $y$, define $I_j^{(i+1)}=[2s^2+2x-1, 2s^2+2y]$. For the rest $2s$ intervals,
\begin{equation*}
I_{m_i+j}^{(i+1)}=
\begin{cases}
[2s(j-1)+1, 2sj], &\text{ when }j\in[s];\\
[4s^2+2sj+1, 4s^2+2s(j+1)], & \text{ when }j\in [s+1, 2s].
\end{cases}
\end{equation*}
So the smallest length $\ell_{i+1}=\min\{\ell_i, N_i-1 \}$. Let $V_{N_{i+1}}=\{(v, j):  v\in V_{N_i}\backslash\{v_\infty\}, j\in [2]\}$. By Fact~\ref{fact_for_multi_ind_cons}, any vertex $v\in V_{N_{i+1}}$ satisfies $|S(t_{i+1}, v)\cap I_j^{(i+1)}|\le 2$ for any $j\in[m_{i+1}]$ and $(N_{i+1}-1)-|S(t_{i+1}, v)\cap(\cup_{j=1}^{m_{i+1}}I_j^{(i+1)})|\le \gamma$. Hence $|V_{N_{i+1}}|\ge 2(|V_{N_i}|-1)$.

Do the analysis above from $t_i$ to $t_{i+1}$ for all $i\in [0, a+b-1]$ inductively, and we get the desired result.
\end{proof}

\begin{lemma}
 Given an admissible $p$ with $p=O(\sqrt n)$. From inductive and factorial constructions listed above we cannot obtain edge labelings with $p$-robustness ratio away from $1$.
\end{lemma}
\begin{proof}
We only prove the cases when $n\equiv 2\mod 4$ and $n\equiv 1\mod 4$. The $n\equiv 3\mod 4$ case is much the same as the first case and we omit the proof.

Since $p=O(\sqrt n)$, there exists some constant $C$ such that $p<C\sqrt n$. We always choose $\varepsilon=20C$ when applying Lemma~\ref{vertex_link_lemma}.

When $n\equiv 2\mod 4$, suppose we have a construction of a supermagic $t$ on $K_n$. Then by Proposition ~\ref{propositioin_large_case} (1), the labeling $t$ is obtained from some $t_B$ on $K_N$ for some $N \equiv 2\mod 4$ by using $\frac{n-N}4$ times of Fact~\ref{fact_for_add_induction} (3).

If such $N\ge\max\{8p, \log n\}$, then for the reason that all the known direct constructions of supermagic labelings, except the factorial constructions, are only for small order $N$, $t_B$ must be a factorial construction on $K_N$ in the form of Fact~\ref{fact_for_factor_cons} (1). Thus, there exist $(N-1)$ disjoint intervals from $[\epsilon_N]$ with lengths $\frac N 2$ such that the preimages of those intervals form a matching decomposition of $K_N$. Then by Proposition~\ref{propositioin_large_case} (1) for taking $\ell=\frac N 2$, $m=N-1$ and $\gamma=0$, there exist $m'=m+\frac{n-N}2$ disjoint $1$-APs $I_1', \ldots, I_{m'}'$ with lengths no smaller than $\frac N 2$, such that for any $v\in V(K_N)$, $S(t, v)\subset \cup_{j=1}^{m'}I_j'$ and $|S(t, v)\cap I_j'|\le 2$ for any $j\in [m']$. By Lemma~\ref{vertex_link_lemma} with $T=V(K_N)$, $r(p, t)\ge 1-\frac {3+\epsilon^2} {2h}$ for any fixed positive integer $h$, which means $r(p, t)=1-o(1)$. In other words, $t$ is of bad $p$-robustness ratio.

If $N<\max\{8p, \log n\}$, then consider the induction relationship tower $t_B=t_0\to t_1\to t_2\cdots\to t_{(n-N)\slash 4}=t$. In this tower,  $t_i$ is supermagic on $K_{N+4i}$ for any $i\in[0, (n-N)\slash 4]$. Denote $j=\min\{i\in [0, (n-N)\slash 4]: N+4i\ge \max\{8p, \log n\}\}$, and denote $t'=t_j$ as the new base case on $K_{n'}$ with $n'=N+4j$. Consider the process from $t'$ to $t$. Apply Proposition~\ref{propositioin_large_case} (1) by letting $\ell=\infty$, $m=0$ and $\gamma=n'-1<\varepsilon\sqrt n$.
Then there exist $m'=\frac{n-n'}2$ disjoint $1$-APs $I_1', \ldots, I_{m'}'\subset [\epsilon_n]$ with lengths no smaller than $2n'+2$, such that for any $v\in V(K_{n'})$, $|S(t, v)\backslash (\cup_{j=1}^{m'}I_j')|\le \gamma$ and $|S(t, v)\cap I_j'|\le 2$ for any $j\in [m']$. Then we can apply Lemma~\ref{vertex_link_lemma} with $T=V(K_{n'})$, and again get that  $r(p, t)=1-o(1)$.

When $n\equiv 1\mod 4$, suppose we have a construction of a supermagic $t$ on $K_n$. Then by Proposition ~\ref{propositioin_large_case} (2), the labeling $t$ is from some directly constructed base case $t_B$ on $K_N$ by inductively applying in total $a$ times of Fact~\ref{fact_for_add_induction} (2) and $b$ times of Fact~\ref{fact_for_multi_ind_cons}. So $b\le log(\frac n N)\le log(n)$.
Notice that we do not have factorial constructions when $n\equiv 1\mod 4$, so $N$ must be small. More specially, $N$ should be smaller than $\max\{2log(n), 10p\}$.

Consider the induction relationship tower $t_B=t_0\to t_1\to t_2\cdots\to t_{a+b}=t$, in which $t_i$ is supermagic on $N_i$ for any $i\in[0, a+b]$ with $N_i=N_{i-1}+4$ or $N_i=2N_{i-1}-1$. Choose the labeling $t'=t_j$ for some $j=\min\{i\in [0, a+b]: N_j\ge \max\{2log(n), 10p\}\}$.
Denote such $N_j$ as $N'$. So $N'<\max\{ 4log(n), 20p\}$. In $t'$, there exists ZERO $1$-APs with $\gamma=N'<20C\sqrt n=\varepsilon \sqrt n$ and $V_{N'}=V(K_{N'})$ such that for any $v\in V_{N'}$, $(N'-1)-|S(t', v)\cap\emptyset|\le \gamma$. Take $t'$ as the new base case and Proposition~\ref{propositioin_large_case} (2) tells us that there exist disjoint $1$-APs $I_1', \ldots, I_{m'}'\subset [\epsilon_n]$ with lengths larger than $4p$, and a new set $V_n\subset V(K_n)$ with $|V_n|\ge log(n)$, such that for any $v\in V_n$, $|S(t, v)\cap I_j'|\le 2$ for any $j\in [m']$ and $(n-1)-|S(t, v)\cap(I_1'\cup\cdots\cup I_{m'}')|\le\gamma$. By Lemma~\ref{vertex_link_lemma}, such a $t$ is also a bad labeling with $p$-robustness ratio $1-o(1)$.
\end{proof}

\bibliographystyle{IEEEtran}
\bibliography{FRC}
\vspace{12pt}

\end{document}